\documentclass[a4paper,11pt,fleqn]{article}
\usepackage[official]{eurosym}
\usepackage{amsmath}
\usepackage{amssymb}
\usepackage{theorem}
\usepackage[usenames,dvipsnames]{pstricks}
\usepackage{euscript}
\usepackage{graphicx}
\usepackage{charter}

\usepackage{comment}

\input alphabet

\newcommand{\email}{}
\definecolor{labelkey}{rgb}{0,0.08,0.45}
\definecolor{refkey}{rgb}{0,0.6,0.0}
\definecolor{Brown}{rgb}{0.45,0.0,0.05}
\definecolor{dgreen}{rgb}{0.00,0.49,0.00}
\definecolor{dblue}{rgb}{0,0.08,0.75}
\RequirePackage[colorlinks,hyperindex]{hyperref} 
\hypersetup{linktocpage=true,citecolor=dblue,linkcolor=dgreen}
\usepackage{cleveref}
\PassOptionsToPackage{normalem}{ulem}
\usepackage{ulem}

\renewcommand{\le}{\ensuremath{\leqslant}}
\renewcommand{\ge}{\ensuremath{\geqslant}}
\newcommand{\minimize}[2]{\ensuremath{\underset{\substack{{#1}}}%
{\text{\rm minimize}}\;\;#2 }}

\newcommand{\scal}[2]{{\left\langle{{#1}\mid{#2}}\right\rangle}}

\newcommand{\emp}{\ensuremath{{\varnothing}}}

\newcommand{\Id}{\ensuremath{\operatorname{I}}}

\newcommand{\RR}{\ensuremath{\mathbb{R}}}

\newcommand{\NN}{\ensuremath{\mathbb N}}

\newcommand{\intdom}{\ensuremath{\text{int\,dom}}}

\newcommand{\dom}{\ensuremath{\text{\rm dom}\,}}
\newcommand{\graph}{\ensuremath{\text{\rm graph}\,}}
\newcommand{\prox}{\ensuremath{\text{\rm prox}}}

\newcommand{\sign}{\ensuremath{\text{\rm sign}}}

\newcommand{\argmind}[2]{\ensuremath{\underset{\substack{{#1}}}%
{\mathrm{argmin}}\;\;#2 }}
\newcommand{\Argmind}[2]{\ensuremath{\underset{\substack{{#1}}}%
{\mathrm{Argmin}}\;\;#2 }}

\newtheorem{theorem}{Theorem}[section]
\newtheorem{lemma}[theorem]{Lemma}

\newtheorem{proposition}[theorem]{Proposition}
\theoremstyle{plain}{\theorembodyfont{\rmfamily}%
}
\theoremstyle{plain}{\theorembodyfont{\rmfamily}%
\newtheorem{assumption}[theorem]{Assumption}}
\theoremstyle{plain}{\theorembodyfont{\rmfamily}%
}
\theoremstyle{plain}{\theorembodyfont{\rmfamily}%
}
\theoremstyle{plain}{\theorembodyfont{\rmfamily}%
}
\theoremstyle{plain}{\theorembodyfont{\rmfamily}%
\newtheorem{remark}[theorem]{Remark}}
\theoremstyle{plain}{\theorembodyfont{\rmfamily}%
\newtheorem{definition}[theorem]{Definition}}
\theoremstyle{plain}{\theorembodyfont{\rmfamily}%
}

\numberwithin{equation}{section}

\begin{document}


\title{\sffamily Variable Metric Forward-Backward Algorithm \\
for Composite Minimization Problems}
\author{
Audrey Repetti$^{\dagger\star}$ and Yves Wiaux$^\star$ 
\\[5mm]
\small
\small $^\dagger$ School of Mathematics and Computer Sciences, Heriot-Watt University, Edinburgh, UK\\
\small $^\star$ School of Engineering and Physical Sciences, Heriot-Watt University, Edinburgh, UK\\
\small \email{\{a.repetti, y.wiaux\}@hw.ac.uk}
}
\date{}

\maketitle
\thispagestyle{empty}

\vskip 8mm

\begin{abstract}
We present a forward-backward-based algorithm to minimize a sum of a differentiable function and a nonsmooth function, both being possibly nonconvex. The main contribution of this work is to consider the challenging case where the nonsmooth function corresponds to a sum of non-convex functions, resulting from composition between a strictly increasing, concave, differentiable function and a convex nonsmooth function. The proposed variable metric Composite Function Forward-Backward algorithm (C2FB) circumvents the explicit, and often challenging, computation of the proximity operator of the composite functions through a majorize-minimize approach. Precisely, each composite function is majorized using a linear approximation of the differentiable function, which allows one to apply the proximity step only to the sum of the nonsmooth functions.  We prove the convergence of the algorithm iterates to a critical point of the objective function leveraging the Kurdyka-\L ojasiewicz inequality.  The convergence is guaranteed even if the proximity operators are computed inexactly, considering relative errors. We show that the proposed approach is a generalization of reweighting methods, with convergence guarantees. In particular, applied to the log-sum function, our algorithm reduces to a generalized version of the celebrated reweighted $\ell_1$ method. Finally, we show through simulations on an image processing problem that the proposed C2FB algorithm necessitates less iterations to converge and leads to better critical points compared with traditional reweighting methods and classic forward-backward algorithms.
\end{abstract}




\section{Introduction}

In this work we consider optimization problems of the form
\begin{equation}	\label{prob:min}
    \minimize{x \in \RR^N} \Big\{ f(x) = h(x) + g(x) \Big\},
\end{equation}
where $h \colon \RR^N \to \RR$ is a Lipschitz differentiable function with constant $\mu>0$, 
and $g \colon \RR^N \to ]-\infty, +\infty]$ is a sum of composite functions as follows
\begin{equation}    \label{def:crit}
    (\forall x \in \RR^N) \quad
    \displaystyle g(x) = \sum_{p=1}^P  (\phi_p \circ \psi_p) (x).
\end{equation}
Composite optimization problems of this form have been studied extensively during the last decades, notably in, e.g., \cite{Burke85, Cartis2011, Drusvyatskiy2016b, Fletcher2009, Lewis2015, Powell83, Powell84, Wright90, Yuan85}. 
In general, all these works rely on the same strategy, consisting in iteratively minimizing approximations to the objective function $f$.

\subsection{Related work}

A first notable example is when $P=1$, $\phi_1$ is the identity function and $g \equiv \psi_1$ is a proper, lower semi-continuous function whose proximity operator can be computed. In this context, a common approach to solve \eqref{prob:min} is the forward-backward (FB) algorithm \cite{Chen97, combettes2005signal, Lions79, Tseng_P_2000_j-siam-control-optim_Modified_fbs}, which alternates between a gradient step on $h$ and a proximity step on $g$. Precisely, at each iteration $k\in \NN$, given the current iterate $x_k \in \RR^N$, the next iterate is defined as
\begin{equation}    \label{eq:FBprox}
    x_{k+1} = \prox_{\gamma_k g} \big( x_k - \gamma_k \nabla h(x_k) \big),
\end{equation}
where $\prox_{\gamma_k g}$ denotes the proximity operator\footnote{The definition of the proximity operator and all the other mathematical definitions and notation used throughout this paper will be given in section~\ref{Sec:Optim}.} of $\gamma_k g$, and $\gamma_k>0$. The convergence of the iterates $(x_k)_{k\in \NN}$ of the FB algorithm to a minimizer of $f$ has been established when both $h$ and $g$ are convex, and choosing $\gamma_k \in ]0, 2/\mu[$ (see e.g. \cite{Burger2016, combettes2005signal}). This result has been extended in \cite{Attouch_Bolte_2011} to the case when both $h$ and $g$ are nonconvex. Assuming that $\gamma_k \in ]0, 1/\mu[$, the authors have proved the convergence of $(x_k)_{k\in \NN}$ to a critical point of $f$, using the Kurdyka-\L ojasiewicz (KL) inequality \cite{Bolte10, Kurdyka94, Lojasiewicz63}. 
However, as many first-order minimization methods, it may suffer from slow convergence \cite{Chen97}. 
Accelerated versions of the FB algorithm have been proposed in the literature, mainly based either on Nesterov' accelerations \cite{beck2009fast, Liang2019} (also called subspace accelerations, or FISTA) or on preconditioning strategies \cite{Chouzenoux13, Chouzenoux_2016, frankel2015}. While the former uses information from the previous iterates, the second aims to improve the step-size $\gamma_k$ at each iteration by introducing a symmetric positive definite (SPD) matrix $A_k \in \RR^{N \times N}$, leading to the following variable metric forward-backward algorithm (VMFB):
\begin{equation}    \label{eq:VMFBprox}
    x_{k+1} = \prox^{\gamma_k^{-1} A_k}_{g} \big( x_k - \gamma_k A_k^{-1} \nabla h(x_k) \big).
\end{equation}
When $A_k$ is chosen to be equal to the identity matrix $\Id_N$, then the basic FB algorithm is recovered. However, as shown in \cite{Chouzenoux13, Repetti2014_icassp}, wiser choices of $A_k$ can drastically accelerate the convergence of the iterates, and even outperform FISTA. In particular, in \cite{Chouzenoux13}, the authors proposed to choose the preconditioning matrices using a majorize-minimize (MM) approach \cite{HunterLange2004, Mairal2013, Sun2017, Wu1983}. Indeed, one can notice that \eqref{eq:VMFBprox} can be equivalently rewritten as an MM algorithm:   
\begin{multline}    \label{eq:VMFB_MM}
    x_{k+1} = \argmind{x \in \RR^N} \Big\{ f_k(x, x_k) = g(x) + h(x_k) + \scal{\nabla h(x_k)}{x-x_k} \\
    + \frac{1}{2\gamma_k} \scal{ x - x_k}{A_k(x-x_k)}  \Big\}.
\end{multline}
The above equation also shows that the VMFB (and FB) algorithm can be interpreted as the proximal regularization of $h$ linearized at the current iterate $x_k$ \cite{beck2009fast, Bolte14}.
When $A_k = \Id_N$, using the descent lemma, it is straightforward to notice that \linebreak $f_k( \cdot, x_k) \colon \RR^N \to ]-\infty, +\infty]$ is a majorant function of $f$ at $x_k$ in the sense that, for every $x \in \RR^N$, $f(x) \le f_k(x, x_k)$ and $f(x_k) = f_k(x_k, x_k)$. In \cite{Chouzenoux13}, the authors proposed to choose $A_k$ to define a more accurate majorant function at each iteration, and proved the convergence of sequences $(x_k)_{k\in \NN}$ generated by \eqref{eq:VMFBprox} to a critical point of $f$ using the KL inequality (when $g$ is assumed to be convex).
However, to the best of our knowledge, there exists no version of the FB algorithm with convergence guarantees able to solve the general composite problem \eqref{prob:min}-\eqref{def:crit} for more general choices of $\phi$, in particular when the proximity operator of $g$ cannot be computed.

During the last years, many optimization methods, mainly based on MM strategies, arose to solve the general composite problem in the case where $P=1$, $\phi_1\equiv \phi$, and $\psi_1\equiv \psi$, i.e. 
\begin{equation}	\label{prob:def_g_gen}
(\forall x \in \RR^N)\quad
g(x) =  (\phi \circ \psi)(x).
\end{equation}
In particular, we can distinguish two main approaches.
The first one consists in majorizing the outer function $\phi$ in \eqref{prob:def_g_gen}. This approach has been investigated in  \cite{Geiping2018, Ochs2015siam}. Precisely, in \cite{Ochs2015siam}, the authors assume that $\psi = (\psi_1, \ldots, \psi_J)$ where $J\le N$ and, for every $j\in \{1, \ldots, J\}$, $\psi_j \colon \RR^N \to \RR$ is convex, and $\phi \colon \RR^P \to \RR$ is coordinate-wise non-decreasing. They also do not require $h$ to be differentiable, but to be proper, lower semi-continuous and convex. In this context, the authors propose to solve \eqref{prob:min}-\eqref{prob:def_g_gen} by defining 
\begin{equation}    \label{ochs2015:algo}
    (\forall k \in \NN)\quad
    x_{k+1} = \argmind{x \in \RR^N} h(x) + q\big( \psi(x), x_k \big),
\end{equation}
where $q\colon \RR^J \to \RR$ is a convex, proper, component-wise non-decreasing majorant function of $\phi$ at $\psi(x_k)$. Using the KL inequality, the authors show that $(x_k)_{k \in \NN}$ converges to a critical point of $f$. However, problem~\eqref{ochs2015:algo} needs to be solved exactly at each iteration, which may not be possible in practice. 
Note that convergence of general MM algorithms in a non-convex setting (not necessarily for composite functions) has also been investigated in \cite{Bolte16}, but it necessitates as well problem~\eqref{ochs2015:algo} to be solved exactly at each iteration to ensure the convergence of $(x_k)_{k \in \NN}$.
In \cite{Geiping2018}, the authors propose a similar approach, using the same update as in \eqref{ochs2015:algo}, but under different assumptions, and with a particular form for the majorant function $q(\cdot, x_k)$. Precisely, for every $k\in \NN$, the authors propose to choose the majorant function as follows
\begin{equation}    \label{geiping:maj}
    (\forall u \in \RR^J) \quad
    q\big( u, x_k \big) = \phi\big(\psi(x_k)\big) + \scal{\nabla \phi \big(\psi(x_k)\big)}{u - \psi(x_k)} + \frac{1}{\gamma} D_d \big( u, \psi(x_k)\big),
\end{equation}
where $\gamma>0$ and $D_d$ is the Bregman distance relative to a Legendre function $d$ \cite{Rockafellar70}. In particular, when $d$ is the usual Euclidean norm squared, then \eqref{geiping:maj} corresponds to a quadratic majorant of $\phi$ at $\psi(x_k)$. In this work, the authors do not assume $h$ to be differentiable, but only proper and lower semi-continuous. The function $\psi$ is assumed to be continuously differentiable, and $\phi$ to be differentiable such that $a d-\phi$ is convex for some constant $a>0$, and $\gamma^{-1} \nabla d - \nabla \phi$ is locally Lipschitz continuous on $\intdom \, d$ (for $d$ strongly convex given by the Bregman distance in \eqref{geiping:maj}). Under these technical assumptions and using the KL inequality, the convergence of $(x_k)_{k\in \NN}$ is then guaranteed. However, both functions $\phi$ and $\psi$ necessitate to be differentiable, and problem~\eqref{ochs2015:algo} must be solved accurately at each iteration (often using sub-iterations).
The second approach to solve the full composite problem~\eqref{prob:min}-\eqref{prob:def_g_gen} consists in using Taylor-like models, and has been investigated in \cite{Drusvyatskiy2016, Ochs2018JOTA}. In both the works, the authors propose to investigate a general problem, consisting in minimizing the function $f$ on $\RR^N$. 
In particular, in \cite{Drusvyatskiy2016}, the authors propose to define the next iterate as 
\begin{equation}    \label{Drus:algo}
    (\forall k \in \NN)\quad
    x_{k+1} = \argmind{x \in \RR^N} f_{x_k}(x),
\end{equation}
where $f_{x_k} \colon \RR^N \to ]-\infty, +\infty]$ is a model function for $f$ at $x_k$, in the sense that there exists a growth function $w \colon [0,+\infty[ \to [0,+\infty[$ such that, for every $x \in \RR^N$, $| f(x) - f_{x_k}(x) | \le w( |x -x_k|)$.
In \cite{Ochs2018JOTA}, the authors propose a modified version of~\eqref{Drus:algo}, adding a Bregman distance $D_{d_k}$ to $f_{x_k}$ that changes at each iteration.
In both the cases, sub-iterations are needed at each iteration. In addition, both the two works have similar convergence guarantees, obtained without the need of the KL inequality. It is interesting to note that, in \cite{Ochs2018JOTA}, the authors pointed out that, at each iteration, the function $f_{x_k}+D_{d_k}$ must be solved accurately to reach asymptotic convergence.

It is worth mentioning that other types of methods have been proposed to minimise composite functions. In particular one can mention the Gauss-Newton methods \cite{Ortega_book_1970}. These methods have been originally proposed to minimize functionals of the particular fom of $\| \psi(\cdot) \|^2$, by solving iteratively approximated problems where $\psi$ is linearized. This approach can also be used to minimize the sum of $\| \psi(\cdot) \|^2$ and an additional term, however traditional Gauss-Newton methods may be unstable when the additional term is non-smooth \cite{Valkonen14}. To circumvent this issue, a Gaussian-Newton method with proximal linearization have been proposed in \cite{Jauhiainen2020}. Nevertheless this approach cannot handle more general composite functions than choosing $\phi = \| \cdot \|^2$. Note that Gauss-Newton methods have also been used to minimize more general composite functions $\phi \circ \psi$ (see e.g. \cite{Burke95}) for $\phi$ convex and $\psi$ continuously differentiable, but without additional term. 




In the current work, we propose a novel algorithm merging the structure of the VMFB algorithm with the MM strategy consisting in approximating the composite function of interest.

\subsection{Proposed approach}

In this work, we will aim to solve problem~\eqref{prob:min}-\eqref{def:crit} where, 
for every $p\in \{1, \ldots,P\}$, $\psi_p \colon \RR^N \to [0,+\infty]$ is convex, proper, lower semi-continuous and Lipschitz continuous on its domain, and $\phi_p \colon [0,+\infty] \to ]-\infty, +\infty]$ is a concave, strictly increasing and differentiable function, such that 
$(\phi_p' \circ \psi_p)$ is Lipschitz-continuous on the domain of $\psi_p$, where $\phi_p'$ denotes the first derivative of $\phi_p$. 

To solve problem~\eqref{prob:min}, we propose to use a VMFB-based algorithm. Basically, as in \eqref{eq:VMFBprox}, the Lipschitz differentiable function $h$ is handled through a gradient step, while the non-smooth term $g$ is handled using a proximal step. However, due to the composite form of $g$ in \eqref{def:crit}, its the proximity operator might not be computable,~either efficiently, or at all. To overcome this difficulty, we propose to replace at each iteration $k\in \NN$, the function $g$ by an approximation denoted by $q( \cdot, x_k) \colon \RR^N \to ]-\infty, +\infty]$. Precisely, this approximation is chosen to be a majorant function of $g$ at $x_k$:
\begin{equation}
\displaystyle (\forall x \in \RR^N) \quad
\begin{cases}
\displaystyle g(x)  \le q(x,x_k) = \sum_{p=1}^P q_p(x,x_k),	\\
g(x_k) = q(x_k, x_k).
\end{cases} 		\label{eq:maj_phi3}
\end{equation}
such that, for every $p \in \{1, \ldots, P\}$, $q_p (\cdot , x_k) \colon \RR^N \to ]-\infty, +\infty]$ is a majorant function of $ (\phi_p \circ \psi_p)$ at $x_k$, i.e.
\begin{equation}
(\forall x \in \RR^N) \quad
\begin{cases}
  (\phi_p \circ \psi_p) (x) \le q_p(x,x_k),	\\
  (\phi_p \circ \psi_p) (x_k) = q_p(x_k,x_k),
\end{cases} 		\label{eq:maj_phi2}
\end{equation}
and is obtained by taking, for every $p\in \{1 ,\ldots, P\}$, the tangent of the concave differentiable function $ \phi_p$ at $\psi_p(x_k)$:
\begin{equation}
(\forall x \in \RR^N) \quad
q_p(x,x_k) 
= 	 (\phi_p \circ \psi_p)(x_k) 
		+  (\phi_p' \circ \psi_p)(x_k) \big( \psi_p(x) - \psi_p(x_k) \big)	.	\label{eq:maj_phi1}	
\end{equation}
%
Then, at each iteration $k\in \NN$, the proposed VMFB algorithm with approximated proximal step reads
\begin{equation}	\label{algo:simple-C2FB-1loop}
    x_{k+1} = \prox_{q(\cdot, x_k)}^{\gamma_{k}^{-1} A_{k}} \Big( {x}_{k} - \gamma_{k} A_{k}^{-1} \nabla h({x}_{k}) \Big),
\end{equation}
where $\gamma_k>0$, and $A_k \in \RR^{N \times N}$ is an SPD matrix. 
As recalled in \eqref{eq:VMFB_MM}, the classic FB (and VMFB) algorithm to solve~\eqref{prob:min} can be seen as the proximal regularization of $h$ linearized at the current iterate $x_k$. 
Interestingly, in the proposed algorithm~\eqref{algo:simple-C2FB-1loop}, the second function $g$ is also (partially) linearized through the linearization of the functions $\phi_p$.

To avoid computing the approximated function $q(.,x_k)$ at each iteration, we propose to fix it for a given finite number of iterations. Precisely, at each iteration $k \in \NN$, the majorant function $q(.,x_k)$ is computed using the current iterate $x_k$, and kept fixed for $I_k \in \NN^*$ VMFB iterations, applied to $h+q(.,x_k)$. Then, the proposed variable metric composite function forward-backward (C2FB) algorithm to solve problem~\eqref{prob:min}-\eqref{def:crit} is given by
\begin{equation}	\label{algo:FB_MM}
\begin{array}{l}
x_0 \in \dom g,	\\
\text{for } k = 0,1, \ldots	\\
\left\lfloor
\begin{array}{l}	
\displaystyle \widetilde{x}_{k,0} = x_k,	\\
\text{for } i = 0, \ldots, I_k-1	\\
\left\lfloor
\begin{array}{l}
\displaystyle 
	\widetilde{x}_{k,i+1} 
		= \prox_{q(\cdot, x_k)}^{\gamma_{k,i}^{-1} A_{k,i}} \Big( \widetilde{x}_{k,i} - \gamma_{k,i} A_{k,i}^{-1} \nabla h(\widetilde{x}_{k,i}) \Big),
\end{array}
\right.	\\[0.1cm]
\displaystyle x_{k+1} = \widetilde{x}_{k,I_k},
\end{array}
\right.
\end{array}
\end{equation}
where, for every $k\in \NN$, and every $i \in \{0, \ldots, I_k-1\}$, $\gamma_{k,i}>0$, and $A_{k,i} \in \RR^{N \times N}$ is an SPD matrix. 
One can notice that in the case when $I_k \to \infty$ for every $k\in \NN$, algorithm~\eqref{algo:FB_MM} can be interpreted as an MM algorithm of the same flavour as those proposed in \cite{Drusvyatskiy2016, Geiping2018, Ochs2015siam, Ochs2018JOTA}. Indeed, in this case, at each iteration $k\in \NN$, we have
\begin{equation}    \label{algo:FB_MM_approx}
    x_{k+1} \approx \argmind{x \in \RR^N} h(x) + q(x, x_k).
\end{equation}

In this work, we prove the convergence of sequences $(x_k)_{k\in \NN}$ generated by the C2FB algorithm given in~\eqref{algo:FB_MM} to a critical point of $f$, using the KL inequality. Subsequently, according to the remark above, we show that algorithm~\eqref{algo:FB_MM_approx} converges to a critical point of $f$, if each sub-problem is solved using VMFB iterates, independently of the number of iterations towards solving each sub-problem.

The remainder of the paper is organized as follows. 
We introduce our notation and give useful definitions of non-convex optimization in \cref{Sec:Optim}. The proposed C2FB method, including an inexact version allowing the proximity operator to be computed inexactly, are given in \cref{Sec:method}. In \cref{Sec:cvgce}, we investigate the asymptotic behaviour of the proposed method, and we give the main convergence result of this work. Particular cases of the proposed approach, including reweighting algorithms, are described in \cref{Sec:examples}. Finally, simulation results on a small image restoration problem are provided in \cref{Sec:Simuls}.

\section{Optimization background}
\label{Sec:Optim}

In this section we give the definitions and notation used throughout the paper. For additional definitions on non-convex optimization, we refer the reader, e.g., to \cite{Rockafellar98}.

\subsection{Analysis notation}

\begin{definition}
Let $f \colon \RR^N \to ]-\infty, +\infty]$. 
\begin{enumerate}
\item
The level set of $f$ at height $\delta \in \RR$ is defined as $\text{lev}_{\le \delta} f = \{ x \in \RR^N | f(x) \le \delta \}$.
\item
The domain of $f$ is defined as $\dom f = \{ x \in \RR^N | f(x)<+\infty \}$.
\item
The function $f$ is proper if its domain is non-empty.
\end{enumerate}
\end{definition}

\begin{definition}
Let $C$ be a subset of $\RR^N$ and $\overline{x} \in \RR^N$. The distance from $\overline{x}$ to $C$ is defined by $\operatorname{dist}(\overline{x}, C) = \inf_{x \in \RR^N} \| \overline{x} - x \|$.
If $C = \emp$, then $\operatorname{dist}(\overline{x},C) = \infty$.
\end{definition}

\begin{definition}
Let $A_1 \in \RR^{N \times N}$ and $A_2 \in \RR^{N \times N}$ be two SPD matrices. 
We denote by $A_1 \succcurlyeq A_2 $ the Loewner partial ordering on $\RR^{N \times N}$, defined as, for every $x \in \RR^N$, $x^\top A_1 x \geqslant x^\top A_2 x$.
The weighted norm associated with $A_1$ is defined as, for every $x \in \RR^N$, $\| x \|_{A_1} = ( x^\top A_1 x)^{1/2}$.
\end{definition}

\subsection{Proximity operator}

The proximity operator, relative to a metric, is defined as follows (see e.g. \cite[Sec.~XV.4]{Urruty_Conv_an} and \cite{Attouch_Bolte_2011, Chouzenoux13, Combettes_Vu_2014}).
\begin{definition}	\label{def:prox}
Let $f \colon \RR^N \to ]-\infty, + \infty]$ be a proper, lower-semicontinuous function. 
Let $A \in \RR^{N \times N}$ be an SPD matrix, and let $\overline{x} \in \RR^N$. 
The proximity operator $\prox_{f}^A \colon \RR^N \rightrightarrows \RR^N$ at $\overline{x}$ of $f$ relative to the metric induced by $A$ is given by $\prox_f^A(\overline{x}) = \Argmind{x \in \RR^N} f(x) + \frac12 \|x - \overline{x} \|^2_A$.
\end{definition}

\begin{remark}\ 
\begin{enumerate}
    \item 
    In the definition of the proximity operator, since $\| \cdot \|_A^2$ is coercive and $f$ is proper and lower-semicontinuous, if $f$ is bounded from below by an affine function, then, for every $\overline{x} \in \RR^N$, $\prox_f^A(x)$ is a non-empty set.
    \item 
    If $f$ is convex, for every $\overline{x} \in \RR^N$, $\prox_f^A(\overline{x})$ is unique. 
    In addition, if $A =\Id_N$, then $\prox_f^{\Id_N} \equiv \prox_f$ is the proximity operator originally defined in \cite{moreau1965proximite}.
\end{enumerate}

\end{remark}

\subsection{Sub-gradients}

\begin{definition}	\label{Def:subdiff}
Let $f \colon \RR^N \to ]-\infty, + \infty]$ and $\overline{x} \in \RR^N$. 
The Fr\'echet sub-differential of $f$ at $\overline{x}$ is denoted by $ \widehat{\partial} f(\overline{x})$, and is given by
\begin{equation}	\label{def:subdiff:frechet}
\widehat{\partial} f(\overline{x}) = \Big\{ \widehat{v}(\overline{x}) \in \RR^N | \underset{y \to \overline{x}, y \neq \overline{x}}{\lim \inf} \dfrac{1}{\|\overline{x}-y\|} \big( f(y) - f(\overline{x}) - \scal{y-\overline{x}}{ \widehat{v}(\overline{x})} \big) \ge 0 \Big\}.
\end{equation}
If $ \overline{x} \not\in \dom f$, then $\widehat{\partial}f(\overline{x}) = \emp$.

\noindent
The limiting sub-differential of $f$ at $\overline{x}$ is denoted by $ \partial f(\overline{x})$, and is given by
\begin{multline}
\partial f(\overline{x}) 
	= \Big\{ v(\overline{x}) \in \RR^N | \exists \big( x_k, \widehat{v}(x_k) \big) \to \big( \overline{x}, v(\overline{x}) \big) \\
	\text{ such that } f(x_k) \to f(\overline{x}) \text{ and } (\forall k \in \NN) \; \widehat{v}(x_k) \in \widehat{\partial}f(x_k) \Big\}.
\end{multline}
\end{definition}

\begin{remark}\ \label{Def:subdiff:Fr2}
\begin{enumerate}
\item		\label{Def:subdiff:Fr2:i}
An equivalent definition of \eqref{def:subdiff:frechet} is given by \cite[Def. 8.3]{Rockafellar98}: $\widehat{\partial} f(\overline{x}) = 
\Big\{ \widehat{v}(\overline{x}) \in \RR^N | 
(\forall x \in \RR^N) \; f(x) \ge f(\overline{x}) + \scal{\widehat{v}(\overline{x})}{x - \overline{x}} + o(|x-\overline{x}|)
\Big\}$

\item		\label{Def:subdiff:Fr2:ii}
A necessary condition for $x^\star \in \RR^N$ to be a minimizer of $f$ is that $x^\star$ is a critical point of $f$, i.e. $0 \in \partial f(x^\star)$. If $f$ is convex, this condition is also sufficient.

\end{enumerate}
\end{remark}

The next results are basic chain rules that will be used throughout the paper.

\begin{proposition}[Basic chain rules \cite{Rockafellar98}]	\label{Prop:chain-rules}
\begin{enumerate}
\item \label{Prop:chain-rules:i}
Let $h \colon \RR^N \to \RR$ be a differentiable function and $g \colon \RR^N \to ]-\infty, + \infty]$, then we have $\partial (h+g) = \nabla h + \partial g$.

\item \label{Prop:chain-rules:ii}
Let, for every $p\in \{1, \ldots, P\}$, $\psi_p \colon \RR^N \to ]-\infty, +\infty]$ be a convex, proper and lower-semicontinuous function. Then we have $\partial \Big( \sum_{p=1}^P \psi_p \Big) = \sum_{p=1}^P \partial \psi_p$.
\end{enumerate}
\end{proposition}

The following closedness property (see e.g \cite[Thm. 8.6]{Rockafellar98}) will be useful to establish the convergence of the proposed C2FB algorithm to a critical point of $f$. This property is often used to investigate convergence of sequences in a nonconvex setting (see e.g. \cite{Attouch_Bolte_2011, Bolte18, Bonettini17, Ochs2019}).
\begin{theorem}	\label{Thm:Rock:contsub}
Let $\big( x_k, t(x_k) \big)_{k \in \NN}$ be a sequence belonging to $\graph \partial f$. If \linebreak $\big( x_k, t(x_k) \big)_{k\in \NN}$ converges to $\big( x^\star, t^\star \big)$, and $(f(x_k))_{k\in \NN}$ converges to $f(x^\star)$, then \linebreak$\big( x^\star, t^\star \big) \in \graph \partial f$.
\end{theorem}

The following proposition will be used to link the sub-gradients of the objective function and its majorant. This result is similar to the one presented in \cite[Lemma~1]{Ochs2015siam}, with slightly different conditions (in particular, we do not assume that the outer function $\phi$ has locally Lipschitz continuous gradient). To the best of our knowledge, this sub-gradient calculation rule has not been presented elsewhere.
\begin{proposition}	\label{Lem:subdiff}
Let $\psi \colon \RR^N \to [0,+\infty]$ be a proper function which is continuous on its domain, and let $\phi \colon [0,+\infty] \to ]-\infty, +\infty]$ be a concave, strictly increasing and differentiable function. 
We further assume that $(\phi' \circ \psi)$ is continuous on its domain. 
Then, for every $\overline{x} \in \RR^N$, we have $\partial (\phi \circ \psi) ( \overline{x}) = (\phi' \circ \psi)(\overline{x}) \partial \psi(\overline{x})$.
\end{proposition}
\begin{proof}
Firstly, let us show the first inclusion, i.e., $\partial (\phi \circ \psi)( \overline{x}) \subset (\phi' \circ \psi)(\overline{x}) \partial \psi(\overline{x})$, for $\overline{x}\in \RR^N$. 
Let $v(\overline{x}) \in \partial \big(\phi \circ \psi\big)(\overline{x})$. 
According to Definition~\ref{Def:subdiff}, there exists $\big( x_k,v(x_k) \big)_{k\in \NN}$ converging to $\big( \overline{x},v(\overline{x}) \big)$, such that $(\phi \circ \psi) (x_k) \to (\phi \circ \psi)(\overline{x})$ and, for every $k\in \NN$, $v(x_k) \in \widehat{\partial} \big(\phi \circ \psi\big) (x_k)$. Thus, according to Remark~\ref{Def:subdiff:Fr2}\ref{Def:subdiff:Fr2:i}, for every $x \in \RR^N$, we have
\begin{equation}	\label{pr:Lem:subdiff:i}
(\phi \circ \psi)(x) \ge (\phi \circ \psi)(x_k) + \scal{v(x_k)}{x-x_k} + o(|x-x_k|).
\end{equation}
Since $\phi$ is a concave and differentiable function, we have, for every $(u_1,u_2) \in [0,+\infty[^2$, $\phi(u_1) - \phi(u_2) \le \phi'(u_2) (u_1 - u_2)$.
Let $u_1 = \psi(x)$ and $u_2 = \psi(x_k)$, then
\begin{equation}
(\phi \circ \psi)(x) - (\phi \circ \psi)(x_k) \le (\phi' \circ \psi)(x_k) (\psi(x) - \psi(x_k)).
\end{equation}
Combining the last inequality with \eqref{pr:Lem:subdiff:i} leads to
\begin{equation}
(\phi' \circ \psi)(x_k) (\psi(x) - \psi(x_k)) \ge \scal{v(x_k)}{x-x_k} + o(|x-x_k|).
\end{equation}
Since $\phi$ is a strictly increasing function, for every $u \in [0,+\infty]$, $\phi'(u) > 0$. Then $(\phi' \circ \psi)(x_k) \neq 0 $, and the last inequality is equivalent to
\begin{equation}
\psi(x) \ge \psi(x_k) + \scal{\big( (\phi' \circ \psi)(x_k)\big)^{-1} v(x_k)}{x-x_k} + o(|x-x_k|).
\end{equation}
Then, by definition of $\widehat{\partial} \psi (x_k)$, we have $\big( (\phi' \circ \psi)(x_k)\big)^{-1} v(x_k) \in \widehat{\partial} \psi (x_k)$.
In addition, since $x_k \to \overline{x}$ and $\psi$ is a continuous function, we have $\psi(x_k) \to \psi(\overline{x})$.
Finally, since $v(x_k) \to v(\overline{x}) \in \RR^N$, $(\phi'  \circ \psi)$ is continuous on $\dom (\phi' \circ \psi)$, and $(\phi' \circ \psi) (x_k)\in ]0,+\infty[$, we have $\big( (\phi' \circ \psi) (x_k) \big)^{-1} v(x_k) \to \big( (\phi' \circ \psi) (\overline{x}) \big)^{-1} v(\overline{x})$.
Therefore, using 
Definition~\ref{Def:subdiff}, we can conclude that 
$\big( (\phi' \circ \psi(\overline{x}) \big)^{-1} v(\overline{x}) \in \partial \psi(\overline{x})$, i.e. $v(\overline{x}) \in (\phi' \circ \psi)(\overline{x}) \partial \psi(\overline{x})$.

We will now show the second inclusion, i.e. $\partial (\phi \circ \psi)(\overline{x}) \supset (\phi' \circ \psi)(\overline{x}) \partial \psi(\overline{x})$, for $\overline{x} \in \RR^N$. 
Let $v(\overline{x}) = (\phi' \circ \psi)(\overline{x}) r(\overline{x})$, with $r(\overline{x}) \in \partial \psi(\overline{x})$. Then, according to Definition~\ref{Def:subdiff}, there exists a sequence $\big(x_k, r(x_k)\big)_{k\in \NN}$ converging to $(\overline{x}, r(\overline{x}))$, such that $\psi(x_k) \to \psi(\overline{x})$ and $r(x_k) \in \widehat{\partial} \psi(x_k)$. 
According to \cite[Prop. 8.5]{Rockafellar98}, for every $k\in \NN$, on a neighbourhood $\Nc(x_k)$ of $x_k$, there exists a differentiable function $\rho_k$ such that $\nabla \rho_k(x_k) = r(x_k)$, $\rho_k(x_k) = \psi(x_k)$, and , for every $x \in \mathcal{N}(x_k) \setminus \{x_k\}$, $\rho_k(x) < \psi(x)$. 
Let $\overline{\rho}_k = (\phi \circ \rho_k)$. Since $\phi$ is differentiable, the function $\overline{\rho}_k$ is also differentiable, and we have
\begin{equation}
\nabla \overline{\rho}_k(x_k) 
= (\phi' \circ \rho_k)(x_k) \nabla \rho_k(x_k) 
= (\phi' \circ \rho_k)(x_k) r(x_k) 
= (\phi' \circ \psi)(x_k) r(x_k).
\end{equation}
In addition, we have $\overline{\rho}_k(x_k) = (\phi \circ \psi)(x_k)$ and, since $\phi$ is a continuous and strictly increasing function, for every $x \in \mathcal{N}(x_k) \setminus \{x_k\}$, $\overline{\rho}_k(x) = (\phi \circ \rho_k)(x) < (\phi \circ \psi)(x)$. 
Then, using again \cite[Prop. 8.5]{Rockafellar98}, we deduce that $v(x_k) := (\phi' \circ \psi)(x_k) r(x_k) \in \widehat{\partial} ( \phi \circ \psi ) (x_k)$.
In addition, since $x_k \to \overline{x}$ and $(\phi \circ \psi)$ is a continuous function, we have $(\phi \circ \psi)(x_k) \to (\phi \circ \psi)(\overline{x})$.
Finally, since $x_k \to \overline{x}$ and $(\phi' \circ \psi)$ is continuous, we have $ (\phi' \circ \psi)(x_k) \to (\phi' \circ \psi)(\overline{x}) $. In addition, since $r(x_k) \to r(\overline{x})$, we have $v(x_k)
\to (\phi' \circ \psi)(\overline{x}) r(\overline{x}) = v(\overline{x})$.
According to Definition~\ref{Def:subdiff}, we can deduce that $ v(\overline{x}) \in \partial (\phi \circ \psi ) (\overline{x})$. 
\end{proof}

\section{Proposed optimization method and assumptions}
\label{Sec:method}

Before giving the assumptions necessary to prove the convergence of the proposed method, we would emphasize that the C2FB algorithm described in~\eqref{algo:FB_MM} can be rewritten using the proximity operator of $\sum_{p=1}^P \lambda_{p,k} \psi_p$ instead of the proximity operator of $q(., x_k)$.

Let, for every $ x \in \RR^N$ and $k\in \NN$,
\begin{equation}	\label{eq:def:lk_lambdak}
\displaystyle l_k(x) = \sum_{p=1}^P \lambda_{p,k} \psi_p(x)
\quad \text{ where } \quad
(\forall p \in \{1, \ldots, P\})\quad
\lambda_{p,k} =  (\phi_p' \circ \psi_p)(x_k).
\end{equation}
Then, according to \eqref{eq:maj_phi1}, for every $k\in \NN$, we have $q( \cdot, x_k) = l_k + C_k$, where $C_k \in \RR$. Consequently, algorithm~\eqref{algo:FB_MM} can be reformulated as follows:
\begin{equation}	\label{algo:FB_MMbis}
\begin{array}{l}
x_0 \in \dom g,	\\
\text{for } k = 0,1, \ldots	\\
\left\lfloor
\begin{array}{l}	
\displaystyle \widetilde{x}_{k,0} = x_k,	\\
\text{for } i = 0, \ldots, I_k-1	\\
\left\lfloor
\begin{array}{l}
\displaystyle \widetilde{x}_{k,i+1} = \prox_{l_k}^{\gamma_{k,i}^{-1} A_{k,i}} \Big( \widetilde{x}_{k,i} - \gamma_{k,i} A_{k,i}^{-1} \nabla h(\widetilde{x}_{k,i}) \Big),
\end{array}
\right.	\\[0.1cm]
\displaystyle x_{k+1} = \widetilde{x}_{k,I_k}.
\end{array}
\right.
\end{array}
\end{equation}
Using the definition of the proximity operator, we can deduce that, for every $k\in \NN$ and $i\in \{0, \ldots, I_k\}$ we have $\widetilde{x}_{k,i} \in \cap_{p=1}^P \dom \psi_p = \dom g$.

We can observe two particular cases of algorithm~\eqref{algo:FB_MMbis}. 
On the one hand, in the particular case when, for every $k\in \NN$, $I_k = 1$, then algorithm~\eqref{algo:FB_MMbis} reads
\begin{equation}	\label{algo:FB_MM1bis}
\begin{array}{l}
x_0 \in \dom g,	\\
\text{for } k = 0,1, \ldots	\\
\left\lfloor
\begin{array}{l}
\displaystyle x_{k+1} = \prox_{l_k}^{\gamma_k^{-1} A_k} \Big( x_k - \gamma_k A_k^{-1} \nabla h(x_k) \Big),
\end{array}
\right.
\end{array}
\end{equation}
where $l_k$ is given by~\eqref{eq:def:lk_lambdak}. 
Algorithm~\eqref{algo:FB_MM1bis} requires to redefine the majorant function $q(\cdot, x_k)$ at each iteration $k\in \NN$, while in algorithm~\eqref{algo:FB_MMbis} the majorant function is fixed for a finite number of iterations $I_k$. 
On the other hand, as emphasized in the introduction, under technical assumptions, in the limit case when, for every $k\in \NN$, $I_k \to \infty$, according to \cite{Chouzenoux13} the sequence $(\widetilde{x}_{k,i})_{i \in \NN}$ converges to a critical point of $h+l_k$. In other words, each inner-loop in algorithm~\eqref{algo:FB_MMbis} corresponds to the VMFB algorithm as defined in \cite{Chouzenoux13}, for minimizing $h+l_k$, leading to algorithm~\eqref{algo:FB_MM_approx}.

\subsection{Assumptions}

In the remainder of this work, we will focus on functions $h$ and $g$ satisfying the following assumptions. 
Examples of functions satisfying the needed assumptions are described in \cref{Sec:examples}.
\begin{assumption}	\label{Ass:global}\
\begin{enumerate}
\item	\label{Ass:global:i}
The function $h \colon \RR^N \to \RR$ is differentiable, and has a $\mu$-Lipschitzian gradient, with $\mu>0$, i.e., for every $(x,x') \in (\RR^N)^2$, $\| \nabla h(x) - \nabla h(x') \| \le \mu \| x-x' \|$.

\item	\label{Ass:global:ii}
For every $p\in \{1, \ldots,P\}$, $\psi_p \colon \RR^N \to [0, +\infty]$ is a convex, proper and lower-semicontinuous function. 
Moreover, it is Lipschitz-continuous on its domain. 

\item	\label{Ass:global:iii}
For every $p\in \{1, \ldots,P\}$, the function $\phi_p \colon  [0, +\infty] \to ]-\infty, +\infty]$ is a concave and strictly increasing function (i.e. $\phi_p'(u)>0$ for every $u \in [0,+\infty[$). Moreover, it is differentiable on $[0,+\infty[$.

\item	\label{Ass:global:ii-iii}
For every $p\in \{1, \ldots,P\}$, $\phi_p' $ is locally Lipschitz continuous on its domain.

\item	\label{Ass:global:iv}
The function $f$ is coercive, i.e. $\lim_{\|x\|\to +\infty} f(x) = +\infty$.

\end{enumerate}
\end{assumption}

The remark below provides comments on \cref{Ass:global}.

\begin{remark}\	\label{remark:glob}
\begin{enumerate}
\item	\label{remark:glob:i}
According to \cref{Ass:global}\ref{Ass:global:i}-\ref{Ass:global:iii}, $f$ is continuous on \linebreak${\dom g = \cap_{p=1}^P \dom \psi_p}$. 
In addition, according to \cref{Ass:global}\ref{Ass:global:iv}, we can deduce that, for every $x \in \dom g$, $\text{lev}_{\le f(x)} f$ is compact (\cite[Prop. 11.12]{Combettes_1996_book}).

\item	\label{remark:glob:ii}
According to \cref{Ass:global}\ref{Ass:global:ii}, there exists $\upsilon >0$ such that, for every $p\in \{1, \ldots, P\}$, $\|r_p(x)\| \le \upsilon$, for every $r_p(x) \in \partial \psi_p(x)$, with $x\in \dom \psi_p$. This assumption is satisfied for simple choices of $\psi_p$ (see \cref{Sec:examples} for examples). 

\item	\label{remark:glob:iii}
According to \cref{Ass:global}\ref{Ass:global:ii}-\ref{Ass:global:iii}, $\phi_p' \circ \psi_p$ is continuous on its domain.

\item	\label{remark:glob:iv}
According to \cref{Ass:global}\ref{Ass:global:ii}-\ref{Ass:global:iii} \cref{remark:glob}-\ref{remark:glob:iii}, for every $k\in \NN$ and $p\in \{1, \ldots, P\}$, the parameter $\lambda_{p,k}>0$ introduced in \eqref{eq:def:lk_lambdak} is well defined.

\item	\label{remark:glob:va}
\cref{Ass:global}\ref{Ass:global:ii-iii} holds if, for every $p\in \{1, \ldots,P\}$, $\phi_p$ is $\mathcal{C}^2$ on its domain.

\item	\label{remark:glob:v}
\cref{Ass:global}\ref{Ass:global:ii-iii} holds if and only if, for every $p\in \{1, \ldots,P\}$, the function $\phi_p' \circ \psi_p$ is Lipschitz continuous on every compact subset of $\RR^N$. 
Thus, under \cref{Ass:global}\ref{Ass:global:ii-iii}, there exists $\eta>0$ such that, for every $x \in \dom g$, and $(x', x'') \in \big(\text{lev}_{\le f(x)} f \big)^2$, $\| (\phi_p' \circ \psi_p)(x') - (\phi_p' \circ \psi_p)(x'') \| \le \eta \| x'-x'' \|$.

\end{enumerate}
\end{remark}

For every $k\in \NN$, the SPD matrices $(A_{k,i})_{0 \le i \le I_k-1}$ are used in practice to accelerate the convergence of usual FB methods. They are chosen using the method proposed in \cite{Chouzenoux13, Chouzenoux_2016, frankel2015}, leveraging an MM approach. We define them as follows:
\begin{assumption}	\label{Ass:maj_quad}
Let, for every $k\in \NN$, $(\widetilde{x}_{k,i})_{0 \le i \le I_k}$ be generated by algorithm~\eqref{algo:FB_MM}.
\begin{enumerate}
\item	\label{Ass:maj_quad:i}
For every $k\in \NN$, we have
\begin{equation}
(\forall x \in \RR^N) \quad
h(x) \le h(\widetilde{x}_{k,i}) + \scal{x-\widetilde{x}_{k,i}}{\nabla h (\widetilde{x}_{k,i})} + \frac12 \| x-\widetilde{x}_{k,i} \|^2_{A_{k,i}}.
\end{equation}
\item	\label{Ass:maj_quad:ii}
There exists $(\underline{\nu},\overline{\nu}) \in ]0,+\infty[^2$ such that, for every $k \in \NN$ and $i \in \{0, \ldots, I_k-1\}$, $\underline{\nu} \Id_N \preccurlyeq A_{k,i} \preccurlyeq \overline{\nu} \Id_N$.
\end{enumerate}
\end{assumption}

\begin{remark}
According to \cite[Lem. 2.1]{Chouzenoux13}, under Assumption~\ref{Ass:global}\ref{Ass:global:i}, \linebreak Assumption~\ref{Ass:maj_quad} is trivially satisfied when choosing, for every $k\in \NN$ and for every $i \in \{0, \ldots, I_k-1\}$, $A_{k,i} = \mu \Id_N$ and $\underline{\nu} = \overline{\nu} = \mu$.  
\end{remark}

The two last assumptions are made to ensure that the step-sizes \linebreak $(\gamma_{k,i})_{k\in \NN, 0 \le i \le I_k-1}$ are bounded, and the inner-iteration numbers $(I_k)_{k\in \NN}$ are finite.
\begin{assumption}	\label{Ass:stepsize}
There exists $(\underline{\gamma},\overline{\gamma}) \in ]0,+\infty[^2$ such that, for every $k \in \NN$ and for every $i \in \{0, \ldots, I_k-1\}$ we have $\underline{\gamma} \le \gamma_{k,i} \le 1-\overline{\gamma}$.
\end{assumption}

\begin{assumption}	\label{Ass:max_subit}
There exists $\overline{I} \in \NN^*$ such that, for every $k\in \NN$, $0 < I_k \le \overline{I} < +\infty$.
\end{assumption}

\subsection{Inexact algorithm}

The proximity operator relative to an arbitrary metric may not have a closed form expression. This is also true for some evolved functions, even when the preconditioning operators are diagonal matrices. For instance when, for every $p\in \{1,\ldots,P\}$, $\psi_p$ is a composition between an $\ell_1$-norm and a non-orthogonal matrix (e.g. to promote sparsity in a redundant dictionary), the computation of the proximity operator of $\psi_p$ is done iteratively \cite{Combettes_Vu_2011}. To circumvent this difficulty, we propose an inexact version of the proposed C2FB algorithm given in~\eqref{algo:FB_MMbis}:
\begin{equation}	\label{algo:FB_MMbis_inexact}
\begin{array}{l}
\text{Let $\alpha \in ]1/2, +\infty[$, $\beta \in ]0,+\infty[$ and } x_0 \in \dom g,	\\
\text{for } k = 0,1, \ldots	\\
\left\lfloor
\begin{array}{l}	
\displaystyle \widetilde{x}_{k,0} = x_k,	\\
\text{for } i = 0, \ldots, I_k-1	\\
\left\lfloor
\begin{array}{l}
\text{find $\widetilde{x}_{k,i+1} \in \RR^N$ and, 
$v_k(\widetilde{x}_{k,i+1}) \in \partial l_k(\widetilde{x}_{k,i+1})$
such that} \\
\displaystyle 
    l_k(\widetilde{x}_{k,i+1})
    + \scal{\widetilde{x}_{k,i+1} - \widetilde{x}_{k,i}}{ \nabla h (\widetilde{x}_{k,i})} \\
    \hfill \displaystyle 
    + \alpha \| \widetilde{x}_{k,i+1} - \widetilde{x}_{k,i} \|_{A_{k,i}}^2
    \le l_k(\widetilde{x}_{k,i})  ,   \\
\displaystyle 
    \| \nabla h (\widetilde{x}_{k,i}) + 
    v_k(\widetilde{x}_{k,i+1})
    \| \le \beta \| \widetilde{x}_{k,i+1} - \widetilde{x}_{k,i} \|_{A_{k,i}}, \\
\end{array}
\right.	\\[0.1cm]
\displaystyle x_{k+1} = \widetilde{x}_{k,I_k}.
\end{array}
\right.
\end{array}
\end{equation}
In algorithm~\eqref{algo:FB_MMbis_inexact}, for every $k\in \NN$ and $i\in \{1, \ldots,I_k-1\}$, $A_{k,i} \in \RR^{N \times N}$ is an SDP matrix satisfying \cref{Ass:maj_quad}, and $\partial l_k = \sum_{p=1}^P \lambda_{p,k} \partial \psi_p$ (see \cref{Prop:chain-rules}). Thus, finding $v_k(\tilde{x}_{k,i+1}) \in \partial l_k(\tilde{x}_{k,i+1}) $ is equivalent to finding, for every $p\in \{1, \ldots, P\}$, $r_p(x_{k+1}) \in \partial \psi_p(x_{k+1})$.

Under our assumptions, algorithm~\eqref{algo:FB_MMbis_inexact} can be viewed as an inexact version of algorithm~\eqref{algo:FB_MMbis} (or equivalently algorithm~\eqref{algo:FB_MM}), where, at each iteration $k\in \NN$, the proximity operator of $l_k$ (or equivalently $\sum_p \lambda_{p,k} \psi_p$) can be computed inexactly (i.e. using sub-iterations). 
The first inequality in algorithm~\eqref{algo:FB_MMbis_inexact} is called \emph{sufficient-decrease condition}, while the second inequality is the \emph{inexact optimality condition}. These two conditions allow to handle possible errors arising when the proximity operator is computed approximately, and are often referred to \emph{relative errors}. Although they are common for FB algorithms in a nonconvex context (see e.g. \cite{Attouch_Bolte_2011, Chouzenoux13, Chouzenoux_2016}), they are more of theoretical interest, showing that the algorithm is robust with respect to inexact computations, than of practical use. 
Note that implementable inexactness conditions for the FB algorithm have been proposed in other works (see e.g. \cite{Villa2013} in the convex context and \cite{Bonettini17} in the non-convex context).

We now show formally that, in particular, sequence $(x_k)_{k\in \NN}$ and, for every $k\in \NN$, $(\widetilde{x}_{k,i})_{ 0 \le i \le I_k}$ generated by algorithm~\eqref{algo:FB_MMbis} satisfy the sufficient decrease condition and the inexact optimality condition of algorithm~\eqref{algo:FB_MMbis_inexact}. This will show that algorithm~\eqref{algo:FB_MMbis_inexact} can be viewed as an inexact version of algorithm~\eqref{algo:FB_MMbis}. \\
%
Let $k\in \NN$ and $i\in \{0, \ldots, I_k-1\}$. Using \cref{def:prox}, we have
\begin{multline}
    l_k(\widetilde{x}_{k,i+1})
    + \scal{\widetilde{x}_{k,i+1} - \widetilde{x}_{k,i}}{ \nabla h (\widetilde{x}_{k,i})}
    + \frac{1}{2\gamma_{k,i}} \| \widetilde{x}_{k,i+1} - \widetilde{x}_{k,i} \|_{A_{k,i}}^2 
    \le l_k(\widetilde{x}_{k,i}).
\end{multline}
According to \cref{Ass:stepsize} the first condition in algorithm~\eqref{algo:FB_MMbis_inexact} (i.e. \emph{sufficient-decrease condition}) is obtained with $\alpha = (1-\overline{\gamma})^{-1}/2$.

The second condition in algorithm~\eqref{algo:FB_MMbis_inexact} (i.e. \emph{inexact optimality condition}) is obtained combining algorithm~\eqref{algo:FB_MMbis} with \cref{Ass:maj_quad}\ref{Ass:maj_quad:ii} and \cref{Ass:stepsize}. 
Indeed, using the variational characterization of the proximity operator, we have, for every $k \in \NN$ and $i\in \{0, \ldots, I_k-1\}$,
\begin{align}
&	\widetilde{x}_{k,i+1} = \prox_{l_k}^{\gamma_{k,i}^{-1} A_{k,i}} \big( \widetilde{x}_{k,i} - \gamma_{k,i} A_{k,i}^{-1} \nabla h(\widetilde{x}_{k,i}) \big)		\nonumber	\\
\Leftrightarrow \quad
&	\widetilde{x}_{k,i} - \gamma_{k,i} A_{k,i}^{-1} \nabla h(\widetilde{x}_{k,i}) - \widetilde{x}_{k,i+1}	 \in  \gamma_{k,i} A_{k,i}^{-1}  \partial l_k(\widetilde{x}_{k,i+1})		\label{eq:charac_prox}
\end{align}
Therefore, 
there exists $v_k(\widetilde{x}_{k,i+1}) \in \partial l_k(\widetilde{x}_{k,i+1})$
such that $v_k(\widetilde{x}_{k,i+1}) = \gamma_{k,i}^{-1} A_{k,i} \big( \widetilde{x}_{k,i} - \widetilde{x}_{k,i+1} \big) - \nabla h(\widetilde{x}_{k,i})$.
Then, using \cref{Ass:maj_quad}\ref{Ass:maj_quad:ii} and \cref{Ass:stepsize}, we obtain
\begin{equation}
    \| \nabla h(\widetilde{x}_{k,i}) +  v_k(\widetilde{x}_{k,i+1}) \|
    = \gamma_{k,i}^{-1} \| A_{k,i} \big( \widetilde{x}_{k,i+1} - \widetilde{x}_{k,i} \big) \|  
    \le \underline{\gamma}^{-1} \sqrt{\overline{\nu}} \| \widetilde{x}_{k,i+1} - \widetilde{x}_{k,i} \|_{A_{k,i}},
\end{equation}
and the inexact optimality condition is obtained with $\beta = \underline{\gamma}^{-1} \sqrt{\overline{\nu}}$.

\section{Convergence analysis}
\label{Sec:cvgce}

\subsection{Descent properties}


In this section, we provide convergence properties~on $(f(x_k))_{k\in \NN}$ for $(x_k)_{k \in \NN}$ generated by algorithm~\eqref{algo:FB_MMbis_inexact}. We also investigate the behaviour of the approximate objective function at the current iteration $k\in \NN$, defined~by
\begin{equation}	\label{Def:f_k}
(\forall x \in \dom g) \quad
f_k(x) = h(x) +  q(x, x_k).
\end{equation}

As a first step, for $k\in \NN$ fixed, we focus on the sub-iterations $i \in \{0, \ldots, I_k\}$ and we investigate the behaviour of $(f_k(\widetilde{x}_{k,i}))_{0 \le i \le I_k}$. In particular, in the following lemma we show that $(f_k(\widetilde{x}_{k,i}))_{0 \le i \le I_k}$ is non-increasing.

\begin{lemma}	\label{Lemma:decr00}
Let $(x_k)_{k \in \NN}$ and, for every $k\in \NN$, $(\widetilde{x}_{k,i})_{ 0 \le i \le I_k}$ be generated by \linebreak algorithm~\eqref{algo:FB_MMbis_inexact}. 
Let $k\in \NN$ and $0 \le i_1 < i_2 \le I_k$. 
Under Assumptions~\ref{Ass:global}, \ref{Ass:maj_quad}, \ref{Ass:stepsize} and \ref{Ass:max_subit}, we have
\begin{equation}
h(\widetilde{x}_{k,i_2}) + l_k(\widetilde{x}_{k,i_2})   
\le
	h(\widetilde{x}_{k,i_1}) + l_k(\widetilde{x}_{k,i_1}) 
        - \overline{\alpha} \sum_{i=i_1}^{i_2-1} \|\widetilde{x}_{k,i+1} - \widetilde{x}_{k,i}\|^2		\label{Lemma:decr00:2} 
\end{equation}
where $\overline{\alpha}>0$, and hence
\begin{equation}
f_k(\widetilde{x}_{k,i_2})
\le	
	f_k(\widetilde{x}_{k,i_1}) 
        - \overline{\alpha} \sum_{i=i_1}^{i_2-1} \|\widetilde{x}_{k,i+1} - \widetilde{x}_{k,i}\|^2	.	\label{Lemma:decr00:2b} 
\end{equation}
\end{lemma}

\begin{proof}
Let $(x_k)_{k \in \NN}$ and, for every $k\in \NN$, $(\widetilde{x}_{k,i})_{ 0 \le i \le I_k}$ be generated by algorithm~\eqref{algo:FB_MMbis_inexact}. 
Using the sufficient decrease condition in algorithm~\eqref{algo:FB_MMbis_inexact} (i.e. the first inequality), we have, for every $k\in \NN$ and $i\in \{0, \ldots, I_k-1\}$, 
\begin{equation}
     \Big( l_k( \widetilde{x}_{k,i+1}) - l_k( \widetilde{x}_{k,i}) \Big)
    + \alpha \| \widetilde{x}_{k,i+1} - \widetilde{x}_{k,i} \|_{A_{k,i}}^2  
    \le - \scal{\widetilde{x}_{k,i+1}-\widetilde{x}_{k,i}}{ \nabla h(\widetilde{x}_{k,i})}.
\end{equation}
Let $0 \le i_1 < i_2 \le I_k$.
Summing on $i\in \{i_1 ,\ldots, i_2-1\}$, we obtain
\begin{multline}    \label{lemma:decr00:proof:i}
    \Big( l_k( \widetilde{x}_{k,i_2}) - l_k( \widetilde{x}_{k,i_1}) \Big)
    + \alpha \sum_{i=i_1}^{i_2-1} \| \widetilde{x}_{k,i+1} - \widetilde{x}_{k,i} \|_{A_{k,i}}^2   \\
    \le - \sum_{i=i_1}^{i_2-1} \scal{\widetilde{x}_{k,i+1}-\widetilde{x}_{k,i}}{ \nabla h(\widetilde{x}_{k,i})}.
\end{multline}
According to \cref{Ass:maj_quad}\ref{Ass:maj_quad:i}, we have
\begin{equation}
    h(\widetilde{x}_{k,i+1}) 
    \le h(\widetilde{x}_{k,i}) + \scal{\widetilde{x}_{k,i+1} - \widetilde{x}_{k,i}}{ \nabla h(\widetilde{x}_{k,i})} + \frac12 \| \widetilde{x}_{k,i+1} - \widetilde{x}_{k,i} \|_{A_{k,i}}^2.
\end{equation}
Summing the last inequality on $i \in \{i_1, \ldots, i_2-1\}$, we obtain
\begin{multline}
    - \sum_{i=i_1}^{i_2-1} \scal{\widetilde{x}_{k,i+1}-\widetilde{x}_{k,i}}{ \nabla h(\widetilde{x}_{k,i})}   \\
    \le \Big( h(\widetilde{x}_{k,i_1}) - h(\widetilde{x}_{k,i_2}) \Big) + \frac12 \sum_{i=i_1}^{i_2-1} \| \widetilde{x}_{k,i+1} - \widetilde{x}_{k,i} \|_{A_{k,i}}^2,
\end{multline}
which, combined with \eqref{lemma:decr00:proof:i}, leads to 
\begin{multline}		\label{lemma:decr00:proof:2}
    h(\widetilde{x}_{k,i_2}) + l_k( \widetilde{x}_{k,i_2})   
    \le h(\widetilde{x}_{k,i_1}) + l_k( \widetilde{x}_{k,i_1}) 
    - (\alpha - \frac12) \sum_{i=i_1}^{i_2-1} \| \widetilde{x}_{k,i+1} - \widetilde{x}_{k,i} \|_{A_{k,i}}^2
\end{multline}
Then, using \cref{Ass:maj_quad}\ref{Ass:maj_quad:ii}, there exists $\overline{\alpha} = \underline{\nu} (\alpha - 1/2) $ such that \eqref{Lemma:decr00:2} is satisfied.

Equation~\eqref{Lemma:decr00:2b} is then obtained by adding the term $\sum_{p=1}^P\Big(  (\phi_p \circ \psi_p)(x_k) - \lambda_{p,k} \psi_p(x_k) \Big)$, constant over $i$, on both sides of~\eqref{lemma:decr00:proof:2}, and using the definitions of function $q( \cdot, x_k)$ (see~\eqref{eq:maj_phi3}-\eqref{eq:maj_phi1}), of function $l_k$ (see~\eqref{eq:def:lk_lambdak}) and of function $f_k$ (see~\eqref{Def:f_k}).

\end{proof}

In the following proposition, we investigate the behaviour of the global sequences \linebreak $\big( f(x_k) \big)_{k\in \NN}$ and $\big( \| x_{k+1} - x_k \|^2 \big)_{k\in \NN}$, where $(x_k)_{k\in \NN}$ is generated by algorithm~\eqref{algo:FB_MMbis_inexact}.

\begin{proposition}	\label{Prop:decr1}
Let $(x_k)_{k \in \NN}$ and, for every $k\in \NN$, $(\widetilde{x}_{k,i})_{ 0 \le i \le I_k}$ be generated by algorithm~\eqref{algo:FB_MMbis_inexact}. 
Under Assumptions~\ref{Ass:global}, \ref{Ass:maj_quad}, \ref{Ass:stepsize} and \ref{Ass:max_subit}, the following assertions hold: 
\begin{enumerate}
\item	\label{Prop:decr1:i}
For every $k\in \NN$
, we have
\begin{equation}
f(x_{k+1}) \le f_k(x_{k+1}) \le f(x_k) - \overline{\alpha} \| \chi_k \|^2,
\label{Prop:decr:f_k} 
\end{equation}
where $\overline{\alpha}>0$ is given in \cref{Lemma:decr00}, and $\chi_k = (\widetilde{x}_{k,i+1} - \widetilde{x}_{k,i})_{0 \le i \le I_k-1} \in \RR^{I_k N}$.

In addition, we have
\begin{equation}	\label{Prop:decr:f-xk} 
f(x_{k+1})  \le  f(x_k) - \overline{\alpha} \overline{I}^{-1} \| x_{k+1} - x_k \|^2.
\end{equation}

\item	\label{Prop:decr1:ii}
$\big( f(x_k) \big)_{k\in \NN}$ is a converging non-increasing sequence. 

\item	\label{Prop:decr1:iii}
We have $\displaystyle \sum_{k\in \NN} \| \chi_k \|^2 < + \infty$ and $\displaystyle \sum_{k\in \NN} \| x_{k+1}-x_k \|^2 < + \infty$, and thus \linebreak
$\displaystyle  \lim_{k\to \infty} \| \chi_k \| = 0$ and $\displaystyle  \lim_{k\to \infty} \| x_{k+1} - x_k \| = 0$.
\end{enumerate}
\end{proposition}

\begin{proof}
\begin{enumerate}
\item
The first inequality~\eqref{Prop:decr:f_k} is a direct consequence of \cref{Lemma:decr00}.
On the one hand, the right-sided inequality in \eqref{Prop:decr:f_k} is obtained by choosing $i_1 = 0$, and $i_2 = I_k$ in \eqref{Lemma:decr00:2b}.  Indeed, noticing that $\widetilde{x}_{k,i_1} = x_k$ , $\widetilde{x}_{k, I_k} = x_{k+1}$ and that $f_k(x_k) = f(x_k)$, we have:
\begin{align}
f_k(\widetilde{x}_{k,I_k}) =
f_k(x_{k+1}) 
&	\le f_k(\widetilde{x}_{k,0}) - \overline{\alpha} \sum_{i=0}^{I_k-1} \| \widetilde{x}_{k,i+1} - \widetilde{x}_{k,i} \|^2		\nonumber	\\
&	\le f(x_k) - \overline{\alpha}  \| \chi_k \|^2.
\end{align}
On the other hand, the left-sided inequality in \eqref{Prop:decr:f_k} is obtained using the majoration condition of function $q( \cdot, x_k)$ (see \eqref{eq:maj_phi3}):
\begin{align}
f_k(x_{k+1})
&	= h(x_{k+1}) + q(x_{k+1}, x_k)	
\ge h(x_{k+1}) + g(x_{k+1})
= f(x_{k+1}).
\end{align}


The second inequality \eqref{Prop:decr:f-xk} is obtained using Jensen's inequality, as
\begin{align} \label{eq:pr:decr1:jensen}
\|x_{k+1} - x_k \|^2
&= \| \sum_{i=0}^{I_k-1} (\widetilde{x}_{k,i+1} - \widetilde{x}_{k,i}) \|^2	\nonumber		\\
&\le I_k \sum_{i=0}^{I_k-1} \| \widetilde{x}_{k,i+1} - \widetilde{x}_{k,i} \|^2
= I_k \| \chi_k\|^2.
\end{align}
The result is then obtained using \cref{Ass:max_subit}.

\item
We deduce directly from \eqref{Prop:decr:f_k} that 
$\big( f(x_k) \big)_{k\in \NN}$ is a non-increasing sequence. 
In addition, combining the fact that, for every $k\in \NN$, $x_k \in \dom g = \dom f$ (according to the definition of the proximity operator) with \cref{remark:glob}\ref{remark:glob:i}, we deduce that $(x_k)_{k\in \NN}$ belongs to a compact subset $E$ of $\text{lev}_{f(x_0)} \subset \dom f$. 
Thus, $f$ being lower bounded (according to \cref{Ass:global} $f$ is continuous on its domain and coercive), $\big(f(x_k)\big)_{k\in \NN}$ converges to a real number $\xi$.

\item
According to \cref{Prop:decr:f_k}, we have $\| \chi_k \|^2 \le \overline{\alpha}^{-1} \Big( f(x_k) - f(x_{k+1}) \Big)$.
Let $K$ be a positive integer. Summing the last inequality from $k=0$ to $K-1$ we have $\sum_{k=0}^{K-1} \| \chi_k \|^2 \le \overline{\alpha}^{-1} \Big( f(x_0) - f(x_K) \Big)$.
Since $(f(x_k))_{k \in \NN}$ is a non-increasing sequence converging to $\xi$, we thus obtain $\sum_{k=0}^{K-1} \| \chi_k \|^2 \le \overline{\alpha}^{-1} \Big( f(x_0) - \xi \Big)$.
The result is then obtained by taking the limit for $K \to +\infty$.

The same arguments can be applied for $(\| x_{k+1} - x_k \|^2)_{k\in \NN}$, leveraging \eqref{Prop:decr:f-xk}.
\end{enumerate}

\end{proof}

Finally, the next proposition will be useful to show that the sequences generated by algorithm~\eqref{algo:FB_MMbis_inexact} are approaching the set of critical points of the global objective function $f$.

\begin{proposition}	\label{prop:critset}
Let $(x_k)_{k \in \NN}$ and, for every $k\in \NN$, $(\widetilde{x}_{k,i})_{ 0 \le i \le I_k}$ be generated by algorithm~\eqref{algo:FB_MMbis_inexact}. 
Under Assumptions~\ref{Ass:global}, \ref{Ass:maj_quad}, \ref{Ass:stepsize} and \ref{Ass:max_subit}, we have, for every $k\in \NN$,
\begin{equation}	\label{prop:critset:subgrad-maj}
\| t(x_{k+1}) \| \le \overline{\beta} \| \chi_k \|,
\end{equation}
where $\overline{\beta} 
> 0$, $\chi_k$ is defined in \cref{Prop:decr1}, and
\begin{equation}	\label{prop:critset:subgrad}
\begin{cases}
\displaystyle t(x_{k+1}) = \nabla h (x_{k+1}) + v_{k+1}(x_{k+1}) \in \partial f(x_{k+1})	,	\\
\displaystyle v_{k+1}(x_{k+1}) = \sum_{p=1}^P \lambda_{p,k+1} r_p(x_{k+1}),
\end{cases}
\end{equation}
with, for every $p\in \{1, \ldots, P\}$, $r_p(x_{k+1}) \in \partial \psi_p(x_{k+1})$ provided by algorithm~\eqref{algo:FB_MMbis_inexact}.
\end{proposition}

\begin{proof}
Firstly, using the chain rule given in \cref{Prop:chain-rules}\ref{Prop:chain-rules:i}, we have $\partial f(x_{k+1}) = \nabla h(x_{k+1}) + \partial g(x_{k+1})$.
In addition, using \cref{Lem:subdiff} and \cref{Prop:chain-rules}\ref{Prop:chain-rules:ii}, we notice that $\partial g(x_{k+1}) = \sum_{p=1}^P \lambda_{p,k+1} \partial \psi_p(x_{k+1})$ where $\lambda_{p,k+1} =  (\phi_p' \circ \psi_p)(x_{k+1})$. Thus, we can deduce that $t(x_{k+1}) $ as defined in \eqref{prop:critset:subgrad} belongs to $ \partial f(x_{k+1})$.

We now show inequality~\eqref{prop:critset:subgrad-maj}.
Let $v_k(x_{k+1}) \in \partial l_k(x_{k+1})$ be defined as in algorithm~\eqref{algo:FB_MMbis_inexact}, i.e. $v_k(x_{k+1}) = \sum_{p=1}^P \lambda_{p,k} r_p(x_{k+1})$, where, for every $p\in \{1, \ldots,P\}$, $r_p(x_{k+1}) \in \partial \psi_p(x_{k+1})$. 
Using the triangular inequality and \cref{remark:glob}\ref{remark:glob:ii}, we have
\begin{align}
\| t(x_{k+1}) \|
& = 		\|  \nabla h(x_{k+1}) + v_k(x_{k+1})  +  \sum_{p=1}^P \lambda_{p,k+1} r_p(x_{k+1}) - v_k(x_{k+1})  \|	\nonumber \\
& \le		\| \nabla h(x_{k+1})  + v_k(x_{k+1}) \| + \| \sum_{p=1}^P (\lambda_{p,k+1} - \lambda_{p,k}) r_p(x_{k+1}) \|	\nonumber \\
& \le		\| \nabla h(x_{k+1})  + v_k(x_{k+1}) \| 	+  \sum_{p=1}^P  \|\lambda_{p,k+1} - \lambda_{p,k} \| \| r_p(x_{k+1}) \|	\nonumber \\
& \le		\| \nabla h(x_{k+1})  + v_k(x_{k+1}) \| 	+  \sum_{p=1}^P \upsilon  \|\lambda_{p,k+1} - \lambda_{p,k} \| .	\label{pr:cvg:maj-t-x-k+1}
\end{align}
Let, for every $i \in \{0, \ldots, I_k-2\}$, $v_k(\widetilde{x}_{k, i+1}) \in \partial l_k(\widetilde{x}_{k,i+1})$ be defined as in algorithm~\eqref{algo:FB_MMbis_inexact}. 
Then, using the fact that $x_{k+1} =\widetilde{x}_{k,I_k} $ and Jensen's inequality, we have
\begin{align}
&	\| \nabla h(x_{k+1}) + v_k(x_{k+1}) \|	^2	
	\le 	\sum_{i=0}^{I_k-1} \| \nabla h(\widetilde{x}_{k,i+1}) + v_k(\widetilde{x}_{k,i+1}) \|^2	\nonumber	\\
&	\quad\quad\quad\quad
	\le 	2
	        \sum_{i=0}^{I_k-1} \Big( \| \nabla h(\widetilde{x}_{k,i+1}) - \nabla h(\widetilde{x}_{k,i}) \|^2 
			+  \| \nabla h(\widetilde{x}_{k,i}) +  v_k(\widetilde{x}_{k,i+1}) \|^2 \Big)	\nonumber	\\
&	\quad\quad\quad\quad
	\le 	2
	        \sum_{i=0}^{I_k-1} \Big( \mu^2 \| \widetilde{x}_{k,i+1} - \widetilde{x}_{k,i} \|^2
			+ \overline{\nu} \beta^2 \|\widetilde{x}_{k,i+1} - \widetilde{x}_{k,i} \|^2 \Big),
\end{align}
where the last majoration is obtained using \cref{Ass:global}\ref{Ass:global:i}, \cref{Ass:maj_quad}\ref{Ass:maj_quad:ii} and the second inequality condition in algorithm~\eqref{algo:FB_MMbis_inexact}. 
Then, by definition of $\chi_k$, we obtain
\begin{equation}	\label{pr:cv:maj_subdiff}
\| \nabla h(x_{k+1}) + v_k(x_{k+1})  \|	
\le \sqrt{2(\mu^2 + \overline{\nu} \beta^2)} \| \chi _k \|.
\end{equation}
In addition, according to \cref{Ass:global}\ref{Ass:global:ii-iii} and \cref{remark:glob}\ref{remark:glob:v}, since $(x_k)_{k\in \NN}$ belongs to the compact set $E$, for every $(k_1, k_2) \in \NN^2$, we have $\| \lambda_{p,k_1} - \lambda_{p,k_2} \| \le \eta \| x_{k_1} - x_{k_2} \|$. Thus, combining this last inequality with~\eqref{pr:cvg:maj-t-x-k+1} and \eqref{pr:cv:maj_subdiff}, we obtain $\| t(x_{k+1}) \|
\le \sqrt{2(\mu^2 + \overline{\nu} \beta^2)} \| \chi _k \| + \sum_{p=1}^P \upsilon \eta \| x_{k+1}-x_k \|$.
Noticing that $\| x_{k+1}-x_k \| \le \| \chi_k \|$, we obtain $\| t(x_{k+1}) \| \le \big( \sqrt{2(\mu^2 + \overline{\nu} \beta^2)} + \upsilon \eta \big) \| \chi _k \|$.
\end{proof}

\subsection{Convergence results}

Before giving our main convergence result, we need to introduce our last assumption, concerning the Kurdyka-\L ojasiewicz inequality \cite{Attouch_etal_2010, Bolte06, Bolte07, Bolte10, Kurdyka94, Lojasiewicz63}.
To this aim, we introduce an additional notation: 
for $\zeta>0$, we denote by $\Phi_\zeta$ the class of concave and continuous functions $\varphi\colon [0, \zeta[ \to [0, +\infty[$ such that $\varphi(0) = 0$, $\varphi$ is $\mathcal{C}^1$ on $]0, \zeta[$, and $\varphi'(u) >0$ for every $u \in ]0,\zeta[$.

\begin{assumption} \label{Ass:KL}
The function $f = h+g$ satisfies the Kurdyka-\L ojasiewicz inequality, i.e., for every $\overline{x} \in \dom \partial f$, 
there exist $\zeta>0$, a neighbourhood $E$ of $\overline{x}$, and a function $\varphi \in \Phi_\zeta$, such that for every $x \in E$ satisfying $f(\overline{x}) < f(x) < f(\overline{x}) + \zeta$, we have
\begin{equation}\label{Ass:KL:in}
\varphi'\big( f(x) - f(\overline{x}) \big) \operatorname{dist}( 0, \partial f(x) ) >1.
\end{equation}
%
%
\end{assumption}


As emphasized, e.g., in \cite{Attouch_etal_2010, Attouch_Bolte_2011, Bolte14}, the KL inequality is satisfied for a wide class of functions, and in particular by sub-analytic, log-exp ans semi-algebraic functions\footnote{A function is semi-algebraic if its graph is a finite union and intersection of sets defined by a finite number of polynomial equalities and inequalities.}.

In order to use \cref{Ass:KL} to investigate the convergence of sequences \linebreak $(x_k)_{k\in \NN}$ generated by algorithm~\eqref{algo:FB_MMbis_inexact}, we need to use the following property, initially introduced in \cite[Lemma 1]{Attouch_Bolte_2008} and \cite[Lemma 6]{Bolte14}.
\begin{lemma}[Uniformized KL property]	\label{Lemma:KL:Bolte}
Let $\Omega$ be a compact subset of $\RR^N$.
Let $f \colon \RR^N \to ]-\infty, +\infty]$ be a proper and lower-semicontinuous function, constant on $\Omega$ and satisfying the KL inequality on $\Omega$.
Then, there exists  $\zeta>0$, $\kappa>0$ and $\varphi \in \Phi_{\zeta}$ such that, for every $\overline{x} \in \Omega$, and for every $x \in \RR^N$ satisfying 
\begin{equation}	\label{Lemma:KL:Bolte:cond}
\begin{cases}
\operatorname{dist}(x,\Omega)<\kappa,	\\
f(\overline{x})< f(x) < f(\overline{x}) + \zeta,
\end{cases}
\end{equation}
we have $\varphi' \big( f(x) - f(\overline{x}) \big) \operatorname{dist} \big( 0, \partial f (x) \big) \ge 1$.

\end{lemma}


We will now investigate the properties of the limit point set. Let $(x_k)_{k \in \NN}$ be a sequence generated by algorithm~\eqref{algo:FB_MMbis_inexact}, and $x_0 \in \dom g$ be  a starting point. The associated set of limit points is defined as $\Omega_{x_0} = \big\{ \overline{x} \in \RR^N \, | \, \exists \text{ an increasing sequence of}$ $\text{ integers } (k_j)_{k\in \NN} \text{ such that } x_{k_j} \underset{j \to \infty}{\to} \overline{x} \big\}$.

\begin{remark}	\label{lemma:characOmega1}
According to \cite[Lemma 5]{Bolte14} and \cref{Prop:decr1}\ref{Prop:decr1:iii}, we have
\begin{enumerate}
\item		\label{lemma:characOmega1:i}
$\displaystyle  \lim_{k \to \infty} \text{dist } \Big( x_k , \Omega_{x_0} \Big) = 0$,
\item		\label{lemma:characOmega1:ii}
$\Omega_{x_0}$ is a non-empty, compact and connected set.
\end{enumerate}
\end{remark}

In addition, following the same arguments as in \cite[Lemma 5]{Bolte14}, we have the following results.

\begin{lemma}[Properties of the limit point set]	\label{lemma:characOmega2}
Let $(x_k)_{k\in \NN}$ be a sequence generated by algorithm~\eqref{algo:FB_MMbis_inexact} with starting point $x_0 \in \dom g$.
The following assertions hold under Assumptions~\ref{Ass:global}, \ref{Ass:maj_quad}, \ref{Ass:stepsize} and~\ref{Ass:max_subit}.
\begin{enumerate}
\item		\label{lemma:characOmega2:i}
If $x^\star \in \Omega_{x_0}$, then $x^\star$ is a critical point of $f$, i.e. $0 \in \partial f (x^\star)$.
\item		\label{lemma:characOmega2:ii}
The objective function $f$ is finite and constant on $\Omega_{x_0}$.
\end{enumerate}

\end{lemma}

\begin{proof}\
\begin{enumerate}
\item
Let $x^\star$ be a limit point of $(x_k)_{k \in \NN}$. 
On the one hand, since for every $k\in \NN$, $x_k \in \dom g$, and $f$ is continuous on $\dom g$ (see \cref{remark:glob}\ref{remark:glob:i}),  there exists a subsequence $(x_{k_j})_{j\in \NN}$ converging to $x^\star$ such that $f(x_{k_j}) \to x^\star$ as $j\to +\infty$. In addition, since $(f(x_k))_{k\in \NN}$ is a convergent and non-increasing sequence (according to \cref{Prop:decr1}\ref{Prop:decr1:ii}), we have $f(x_k) \to f(x^\star)$ as $k \to \infty$.
On the other hand, using \cref{Prop:decr1}\ref{Prop:decr1:iii} and \cref{prop:critset}, we have $t(x_{k}) \in \partial f(x_{k})$ and $t(x_k) \to 0$ as $k \to \infty$. 
So, using the closedness property of~$\partial f$ given in \cref{Thm:Rock:contsub}, we have $0 \in \partial f(x^\star)$, and thus $x^\star$ is a critical point~of~$f$.

\item
On the one hand, according to \cref{Prop:decr1}\ref{Prop:decr1:ii}, there exist $\xi \in \RR$ such that $f(x_k) $ converges to $\xi$ as $k\to +\infty$. 
On the other hand, let $x^\star \in \Omega_{x_0}$, i.e. there exists a subsequence $(x_{k_j})_{j \in \NN}$ of $(x_k)_{k\in \NN}$ converging to $x^\star$ as $k\to +\infty$. 
As seen in the proof of statement (i), we have $f(x_k) \to f(x^\star)$ as $k \to +\infty$, thus $f(x^\star) = \xi$. 
Therefore, the restriction of $f$ to $\Omega_{x_0}$ is constant, and equal to $\xi$.
\end{enumerate}
\end{proof}



The next theorem is our main convergence result, analysing the convergence of the sequences generated by the C2FB algorithm given in~\eqref{algo:FB_MMbis_inexact}.

\begin{theorem} \label{thm:cvgce}
Let $(x_k)_{k\in \NN}$ be a sequence generated by algorithm~\eqref{algo:FB_MMbis_inexact}. 
Under Assumptions~\ref{Ass:global}, \ref{Ass:maj_quad}, \ref{Ass:stepsize}, \ref{Ass:max_subit} and~\ref{Ass:KL}, the following holds.
\begin{enumerate}
\item
$(x_k)_{k\in \NN}$ has a finite length, i.e. $\displaystyle \sum_{k=0}^{+\infty} \| x_{k+1} - x_k \| < +\infty$.
\item
$(x_k)_{k\in \NN}$ converges to a critical point $x^\star$ of $f$.
\end{enumerate}
\end{theorem}

\begin{proof}\

\begin{enumerate}
\item
First, in order to apply the result given in \cref{Lemma:KL:Bolte}, we need to show that there exists $k^\star \in \NN$ such that, for every $k > k^\star$, $x_k$ satisfies condition~\eqref{Lemma:KL:Bolte:cond} for $\Omega_{x_0}$.
According to \cref{lemma:characOmega1}\ref{lemma:characOmega1:ii}, $\Omega_{x_0}$ is non-empty. So there exists a convergent subsequence $(x_{k_j})_{j\in \NN}$, i.e. there exists $x^\star$ such that $x_{k_j} \to x^\star$ as $j \to +\infty$. 
In addition, using the same arguments as in the proof of \cref{lemma:characOmega2}, 
we can deduce that $\lim_{k\to +\infty} f(x_k) = f(x^\star)$.
%
On the one hand, $(f(x_k))_{k\in \NN}$ being a convergent and a non-increasing sequence, we have, for every $\zeta>0$, 
\begin{equation}
(\exists \overline{k} \in \NN)(\forall k>\overline{k})\quad
f(x^\star) < f(x_k) < f(x^\star) + \zeta.
\end{equation}
On the other hand, according to \cref{lemma:characOmega1}\ref{lemma:characOmega1:i}, for every $\kappa>0$, there exists $\overline{\overline{k}}\in \NN$ such that, for every $k > \overline{\overline{k}}$, $\operatorname{dist} \big( x_k, \Omega_{x_0} \big) < \kappa$.
So, there exists~$k^\star = \max \{ \overline{k}, \overline{\overline{k}} \}$ such that, for every $k>k^\star$, $x_k$ satisfies condition~\eqref{Lemma:KL:Bolte:cond}.
According to \cref{lemma:characOmega1}\ref{lemma:characOmega1:ii}, $\Omega_{x_0}$ is compact, and using \cref{lemma:characOmega2}\ref{lemma:characOmega2:ii}, $f$ is constant on $\Omega_{x_0}$, thus we can apply \cref{Lemma:KL:Bolte}. Then, for every $k>k^\star$, we have
\begin{equation}
\varphi' \Big( f(x_k) - f(x^\star) \Big) \operatorname{dist} \big( 0, \partial f(x_k) \big) \ge 1.
\end{equation}
Combining this inequality with \cref{prop:critset}, we obtain
\begin{align}	\label{pr:thm:maj-subdiff}
\varphi' \Big( f(x_k) - f(x^\star) \Big)
\ge \|t(x_{k})\|^{-1} 	
&\ge \overline{\beta}^{-1} \|\chi_{k-1} \|^{-1}.
\end{align}
In addition, since $\varphi$ is concave, we have, for every $(u_1,u_2) \in [0,+\zeta[^2$, $\varphi(u_1) - \varphi(u_2) \ge \varphi'(u_1) (u_1 - u_2)$. Then, by taking $u_1 = f(x_k) - f(x^\star)$ and $u_2 = f(x_{k+1}) - f(x^\star)$, we obtain
\begin{multline}
\varphi \Big( f(x_k) - f(x^\star) \Big) - \varphi \Big( f(x_{k+1}) - f(x^\star) \Big) 	\\
\ge \varphi' \Big( f(x_k) - f(x^\star) \Big) \Big(f(x_k) - f(x_{k+1}) \Big).	
\end{multline}
Let $\Delta_k = \varphi \Big( f(x_k) - f(x^\star) \Big) - \varphi \Big( f(x_{k+1}) - f(x^\star) \Big)$. 
Combining the last two inequalities leads to $f(x_k) - f(x_{k+1}) \le 	\overline{\beta} \| \chi_{k-1} \| \Delta_k$. In addition, using \cref{Prop:decr1}\ref{Prop:decr1:i}, we obtain $\overline{\alpha} \| \chi_k \|^2  \le 	f(x_k) - f(x_{k+1})$.
Thus,
\begin{equation}
\| \chi_k \| 
\le  \Big(\overline{\alpha}^{-1} \overline{\beta} \Delta_k  \| \chi_{k-1} \| \Big)^{1/2}
\le \frac12   \Big( {\overline{\alpha}^{-1}\overline{\beta}}  \Delta_k + \|  \chi_{k-1} \| \Big),
\end{equation}
where the last majoration is obtained using the fact that, for every $(u_1, u_2) \in [0,+\infty[^2$, $\sqrt{u_1 u_2} \le (u_1 + u_2)/2$.

Summing the last inequality from $k^\star+1$ to $K>k^*$, we obtain
\begin{align} \label{eq:pr:thm:maj_chik}
2 \sum_{k=k^\star+1}^K \| \chi_k \| 
&\le 	{\overline{\alpha}^{-1} \overline{\beta}} \sum_{k=k^\star+1}^K  \Delta_k + \sum_{k=k^\star+1}^K \| \chi_{k-1} \|		\nonumber		\\
&\le 	{\overline{\alpha}^{-1} \overline{\beta}} \sum_{k=k^\star+1}^K   \Delta_k + \| \chi_{k^*} \| + \sum_{k=k^\star+1}^K \| \chi_k \| .
\end{align}
In addition, we have
\begin{align}	\label{eq:pr:thm:maj_chik2}
\sum_{k=k^\star+1}^K  \Delta_k
&= 	\varphi \Big( f(x_{k^*+1}) - f(x^\star) \Big)
		- \varphi \Big( f(x_{K+1}) - f(x^\star) \Big)		\nonumber	\\
&\le 	\varphi \Big( f(x_{k^*+1}) - f(x^\star) \Big),
\end{align}
where the last inequality is obtained using the fact that, for every $u\in [0,\zeta[$, $\varphi(u) \ge 0$. 
Therefore, combining \eqref{eq:pr:thm:maj_chik} and \eqref{eq:pr:thm:maj_chik2}, we obtain $\sum_{k=k^\star+1}^K \| \chi_k \| \le \overline{\alpha}^{-1} \overline{\beta} \varphi \Big( f(x_{k^*+1}) - f(x^\star) \Big) + \| \chi_{k^*} \|$.
Since $ (x_{k+1}-x_k) = \sum_{i=0}^{I_k-1} (\widetilde{x}_{k,i+1} - \widetilde{x}_{k,i}) $, $\chi_k = (\widetilde{x}_{k,i+1} - \widetilde{x}_{k,i})_{0 \le i \le I_k-1}$, using Jensen's inequality as in \eqref{eq:pr:decr1:jensen} and \cref{Ass:max_subit} we have $\| x_{k+1} - x_k \| \le \sqrt{\overline{I}} \| \chi_k \|$. So we conclude that $(x_k)_{k\in \NN}$ has finite length.

\item
Since $(x_k)_{k\in \NN}$ has finite length, it is a Cauchy sequence, hence converging to a point $x^\star \in \Omega_{x_0}$. Then the result follows directly from \cref{lemma:characOmega2}\ref{lemma:characOmega2:i}.

\end{enumerate}

\end{proof}

\begin{remark}
The proof of \cref{thm:cvgce} is very similar to the one of \cite[Theorem~1]{Bolte14}. However, the later cannot directly be applied to algorithm~\eqref{algo:FB_MMbis_inexact} due to the sub-iterations on $i\in \{0, \ldots, I_k-1\}$, where the majorant function of $g$ is fixed. 
Precisely, in \eqref{pr:thm:maj-subdiff}, the bound used on the sub-gradient of $f$ at $x_{k}$ depends on the sub-iterates $(\widetilde{x}_{k,i})_{0 \le i \le I_k}$ (see \cref{prop:critset}). This prevents the use of the proof of \cite[Theorem~1]{Bolte14} that only relies on the main sequence $(x_k)_{k\in \NN}$.
\end{remark}

\section{Particular cases of the proposed method}
\label{Sec:examples}

In this section we describe particular cases for the proposed C2FB algorithm, paying attention to the assumptions necessary to ensure that the convergence of Theorem~\ref{thm:cvgce} holds.

\subsection{Variable Metric Forward Backward algorithm}

The proposed algorithm boils down to the VMFB developed in \cite{Chouzenoux13} when choosing, for every $k\in \NN$, $I_k=1$, $P=1$, $\phi_1 = \operatorname{Id}$, and $\psi_1 = \psi$ being a proper and Lipschitz-continuous function.
In this case, we have $g= \psi$ and 
the global minimization problem is of the form
\begin{equation}
\minimize{x \in \RR^N} h(x) + \psi(x).
\end{equation}
In this particular case, we have, for every $x \in \RR^N$, $q(x,x_k) = \psi(x)$,
and consequently, algorithms~\eqref{algo:FB_MM} reduces to
\begin{equation}	\label{algo:basic_VMFB}
\begin{array}{l}
x_0 \in \dom g,	\\
\text{for } k = 0,1, \ldots	\\
\left\lfloor
\begin{array}{l}
\displaystyle x_{k+1} = \prox_{\psi}^{\gamma_k^{-1} A_k} ( x_k - \gamma_k A_k^{-1} \nabla h(x_k)).
\end{array}
\right.
\end{array}
\end{equation}
We would emphasize that the assumptions on $\psi$ to ensure convergence of \eqref{algo:basic_VMFB} are different than in \cite{Chouzenoux13}. 
On the one hand, in \cite{Chouzenoux13} the function $\psi$ only needs to be continuous on its domain (while in Assumption~\ref{Ass:global:ii} we assume that $\psi$ is Lipschitz-continuous on its domain). 
Other works present convergence results for algorithm~\eqref{algo:basic_VMFB} when $\psi$ is non-convex either in the case without variable metric \cite{Attouch_Bolte_2011} or for alternating minimization \cite{Chouzenoux_2016}. 
On the other hand, in \cite{Chouzenoux13} the step-sizes $(\gamma_k)_{k\in \NN}$ are chosen such that, there exists $(\underline{\gamma}, \overline{\gamma}) \in ]0,+\infty[^2$ such that, for every $k\in \NN$, $\underline{\gamma} \le \gamma_k \le 2-\overline{\gamma}$, while in Assumption~\ref{Ass:stepsize} it is assumed that $\underline{\gamma} \le \gamma_k \le 1-\overline{\gamma}$.

\subsection{Reweighted algorithms}

The second example is interesting in computational imaging since it is related to iteratively reweighted algorithms \cite{Ochs2015siam}. 

We consider problem~\eqref{prob:min}-\eqref{def:crit} with, for every $p\in \{1, \ldots, P\}$, $\psi_p$ satisfying Assumption~\ref{Ass:global} and, for every $u \in [0,+\infty[$, $\phi_p(u) = \theta \log( u + \varepsilon)$,
where $\theta>0$ and $\varepsilon>0$. In this case, we have 
\begin{equation}
(\forall x \in \RR^N) \quad
\displaystyle g(x) = \theta \sum_{p=1}^P  \log( \psi_p(x) + \varepsilon),
\end{equation}
and, for every $k \in \NN$ and $p \in \{1, \ldots, P\}$, $q_p(x,x_k) =  \theta \log( \psi_p(x_k) + \varepsilon) +  \lambda_{p,k} \big( \psi_p(x) - \psi_p(x_k) \big)$,
with $\lambda_{p,k} = \theta  {\big( \psi_p(x_k)+\varepsilon \big)^{-1}}$.

In this context, the proposed C2FB algorithm given in \eqref{algo:FB_MMbis} reduces to
\begin{equation}	\label{algo:rw_FB}
\begin{array}{l}
x_0 \in \dom g,	\\
\text{for } k = 0,1, \ldots	\\
\left\lfloor
\begin{array}{l}	
\text{for } p = 1, \ldots, P	\\
\left\lfloor
\begin{array}{l}	
\displaystyle \lambda_{p,k} = \theta \big( \psi_p(x_k)+\varepsilon \big)^{-1},	\\[0.1cm]
\end{array}
\right. \\[0.2cm]
\displaystyle \widetilde{x}_{k,0} = x_k,	\\
\text{for } i = 0, \ldots, I_k-1	\\
\left\lfloor
\begin{array}{l}
\displaystyle \widetilde{x}_{k,i+1} = \prox_{l_k}^{\gamma_{k,i}^{-1} A_{k,i}} ( \widetilde{x}_{k,i} - \gamma_{k,i} A_{k,i}^{-1} \nabla h(\widetilde{x}_{k,i})),
\end{array}
\right.	\\[0.1cm]
\displaystyle x_{k+1} = \widetilde{x}_{k,I_k}.
\end{array}
\right.
\end{array}
\end{equation}

\subsubsection{Reweighted $\ell_1$ algorithm for log-sum penalization}
\label{Sssec:logsum}

Let us consider the following particular case. For every $p\in \{1, \ldots, P\}$, 
\begin{equation}    \label{eq:rwl1:psin}
    (\forall x \in \RR^N)\quad
    \psi_p(x) = | [Wx]^{(p)} |,
\end{equation}
where $W \colon \RR^N \to \RR^P$ is a linear operator (e.g. wavelet transform \cite{Mallat_book}), and $[\cdot]^{(p)}$ denotes the $p$-th component of its argument. Then we have, for every $k\in \NN$, $\lambda_{p,k} = {\theta}{\big( |[Wx_k]^{(p)}| + \varepsilon \big)^{-1}}$.
In this context, the function $g$ corresponds to a log-sum penalization composed with a linear operator:
\begin{equation}    \label{def:log-sum}
    (\forall x \in \RR^N) \quad
    g(x) = \sum_{p=1}^P \theta \log \big( |[Wx]^{(p)}| + \varepsilon \big),
\end{equation}
and
algorithm~\eqref{algo:rw_FB} reduces to the re-weighted $\ell_1$ algorithm, initially proposed in \cite{candes2006compressive} for $W=\Id_N$, where each sub-problem is solved using a VMFB algorithm. 
Let $\operatorname{Diag}(\cdot)$ be the operator giving the diagonal matrix whose diagonal elements are given by its argument.
The resulting algorithm reads
\begin{equation}	\label{algo:rL1_FB}
\begin{array}{l}
x_0 \in \dom g,	\\
\text{for } k = 0,1, \ldots	\\
\left\lfloor
\begin{array}{l}
\Lambda_k = \operatorname{Diag} \left( \Big( \frac{\theta}{|[Wx_k]^{(p)}| + \varepsilon} \Big)_{1 \le p \le P} \right), \\[0.2cm]
\displaystyle \widetilde{x}_{k,0} = x_k,	\\
\text{for } i = 0, \ldots, I_k-1	\\
\left\lfloor
\begin{array}{l}
\displaystyle \widetilde{x}_{k,i+1} = \prox_{\|\Lambda_k W \; \cdot \|_1}^{\gamma_{k,i}^{-1} A_{k,i}} ( \widetilde{x}_{k,i} - \gamma_{k,i} A_{k,i}^{-1} \nabla h(\widetilde{x}_{k,i})),
\end{array}
\right.	\\[0.1cm]
\displaystyle x_{k+1} = \widetilde{x}_{k,I_k}.
\end{array}
\right.
\end{array}
\end{equation}

Note that \eqref{def:log-sum} with $W = \Id_n$ is known as the log-sum penalization and its proximity operator has an explicit form \cite{Chierchia_proxrep}. Let $A= \operatorname{Diag} \big( (a^{(n)})_{1 \le n \le N}\big) \in [0,+\infty[^{N \times N}$ be a diagonal SDP matrix. Then we have 
\begin{multline}    \label{eq:prox_logsum}
    (\forall x \in \RR^N)\quad
    \prox_{g}^{A}(x) = (\overline{x}^{(n)})_{1 \le n \le N} 
    \quad \text{where} \quad
    (\forall n \in \{1, \ldots,N\})\\
    \overline{x}^{(n)} = 
    \begin{cases}
    0,  &   \text{if } |x^{(n)}| < \sqrt{4a^{(n)}} - \varepsilon,   \\
    \max\Big\{ 0, \sign(x^{(n)}) \frac{|x^{(n)}| - \varepsilon + \sqrt{(|x^{(n)}|+\varepsilon)^2 - 4a^{(n)}}}{2} \Big\},
        &   \text{if } |x^{(n)}| = \sqrt{4a^{(n)}} - \varepsilon,   \\
    \sign(x^{(n)}) \frac{|x^{(n)}| - \varepsilon + \sqrt{(|x^{(n)}|+\varepsilon)^2 - 4a^{(n)}}}{2},
        &   \text{otherwise}.
    \end{cases}
\end{multline}
Consequently, in this context, one can apply the classic VMFB algorithm, for diagonal matrices $(A_k)_{k\in \NN}$, directly to minimize $f $. 


\subsubsection{Cauchy penalization}

Similarly to the reweighting $\ell_1$ algorithm described above, another particular case is when choosing, for every $p\in \{1, \ldots,P\}$ and $x \in \RR^N$, $\psi_p(x) = \big( [Wx]^{(p)} \big)^2$, where $W \colon \RR^N \to \RR^P$ is a linear operator.
In this case, $g$ corresponds to the Cauchy penalization composed with a linear operator:
\begin{equation}
    (\forall x \in \RR^N)\quad
    g(x) =  \sum_{p=1}^P \theta \log \Big( \big( [Wx]^{(p)} \big)^2 + \varepsilon \Big),
\end{equation}
and, for every $k\in \NN$ and $p \in \{1, \ldots, P\}$, the weights in algorithm~\eqref{algo:rw_FB} are given by $\lambda_{p,k} = \theta {\big( ([Wx]^{(p)} )^2 + \varepsilon \big)^{-1}}$.


\subsection{Nonconvex norms} \label{Ssec:lpnorm}

In this section we show that our method can be used to handle nonconvex $\ell_\rho^\rho$-norms, where $\rho \in ]0,1[$, defined by
\begin{equation}    \label{def:lp}
    (\forall x \in \RR^N)\quad
    \ell_\rho(x)^\rho =  \sum_{n=1}^N |x^{(n)}|^\rho .
\end{equation}
The proximity operator (relative to the Euclidean norm) of $\ell_\rho^\rho$, relative to the metric induced by the diagonal matrix $A=\text{Diag}\big( (a^{(n)})_{1 \le n \le N} \big)$, has an explicit formula, however it necessitates to find roots of polynomial equations \cite{Bredies2008}: 
\begin{multline}    \label{eq:prox_lp}
    (\forall x \in \RR^N)\quad
    \prox^A_{ \ell_\rho^\rho}({x}) = (\overline{x}^{(n)})_{1 \le n \le N} \\
    \text{where} \quad
    (\forall n \in \{1, \ldots,N\})\quad
    \overline{x}^{(n)}
    = \begin{cases}
    0, & \text{if }  |x^{(n)}|^{\rho-2} > \frac{a^{(n)}}{2-\rho} \Big( 2 \frac{1-\rho}{2-\rho} \Big)^{1-\rho} \\
    t^{(n)} x^{(n)}, &   \text{otherwise,}
    \end{cases}
\end{multline}
where $t^{(n)}>0$ is such that $(a^{(n)})^{-1} |x|^{\rho-2} (\rho-1) (t^{(n)})^{\rho-1} + t -1 = 0$. In practice $(t^{(n)})_{1 \le n \le N}$ is found approximately using the Newton method \cite{Chierchia_proxrep}.

The proposed approach cannot handle directly the $\ell_\rho^\rho$ function, and we need to introduce an approximation of it. Indeed, to define the $\ell_\rho^\rho$ norm as a composition of functions, we take, for every $p\in \{1, \ldots,P\}$, $\psi_p$ to be the function defined in \eqref{eq:rwl1:psin}. In this case, to obtain the $\ell_\rho^\rho$ function defined in \eqref{def:lp}, we should choose, for every $p\in \{1 , \ldots, P\}$ and $u\in [0,+\infty[$, $\phi_p(u) = \theta u^\rho $, where $\theta>0$ is a regularization parameter. However this function is not differentiable at $0$, hence violating \cref{Ass:global}\ref{Ass:global:iii}. Instead we propose to choose $\phi_p(u) = \theta \big( (u+\varepsilon)^\rho - \varepsilon^\rho \big)$, where $\varepsilon>0$. 
Then, $g$ is chosen to be a slightly modified version of the nonconvex $\ell_\rho$-norm, denoted by $\widetilde{\ell}_{\rho,\varepsilon}$, composed with a linear operator $W \colon \RR^N \to \RR^P$:
\begin{equation}    \label{def:mod_ellp}
    (\forall x \in \RR^N)\quad
    g(x) = \theta \widetilde{\ell}_{\rho,\varepsilon}(Wx)^\rho =  \theta \sum_{p=1}^P \Big( \big(|[Wx]^{(p)}| + \varepsilon\big)^\rho - \varepsilon^\rho \Big).
\end{equation}
In this context, for every $p \in \{1, \ldots,P\}$, $k\in \NN$ and $x \in \RR^N$, we have $q_p(x,x_k) = \theta \Big( \big( |[W x_k]^{(p)}| + \varepsilon\big)^\rho - \varepsilon^\rho \Big) + \lambda_{p,k} \Big( |[W x]^{(p)}|- |[W x_k]^{(p)}| \Big)$, 
where $\lambda_{p,k} = \theta \rho \big( |[W x_k]^{(p)}| + \varepsilon \big)^{\rho-1}$.
Therefore, the proposed algorithm given in \eqref{algo:FB_MMbis} reduces to
\begin{equation}	\label{algo:lp_FB}
\begin{array}{l}
x_0 \in \dom g,	\\
\text{for } k = 0,1, \ldots	\\
\left\lfloor
\begin{array}{l}	
\Lambda_k = \text{Diag} \Big( \big( \theta \rho ( |[Wx_k]^{(p)}| + \varepsilon )^{\rho - 1} \big)_{1 \le p \le P} \Big)   \\
\displaystyle \widetilde{x}_{k,0} = x_k,	\\
\text{for } i = 0, \ldots, I_k-1	\\
\left\lfloor
\begin{array}{l}
\displaystyle \widetilde{x}_{k,i+1} = \prox_{\| \Lambda_k W \cdot \|_1}^{\gamma_{k,i}^{-1} A_{k,i}} ( \widetilde{x}_{k,i} - \gamma_{k,i} A_{k,i}^{-1} \nabla h(\widetilde{x}_{k,i})),
\end{array}
\right.	\\[0.1cm]
\displaystyle x_{k+1} = \widetilde{x}_{k,I_k}.
\end{array}
\right.
\end{array}
\end{equation}


\section{Simulations}
\label{Sec:Simuls}

Many imaging problems such as reconstruction, restoration, inpainting, etc., can be formulated as inverse problems. In this context, the objective is to find an estimate $x^\star \in \RR^N$ of an original unknown image $\overline{x} \in \RR^N$ from degraded observations $y\in \RR^M$, given by $y=H\overline{x}+b$,
where $H \colon \RR^N \to \RR^M $ is a linear observation operator, and $b\in \RR^M$ is a realization of an additive independent identically distributed (i.i.d.) random noise. When the random noise is normally distributed, with zero-mean, a common approach to find $x^\star$ is to define it as the minimizer of a penalized least-squares criterion, i.e. solve \eqref{prob:min} with, for every $x \in \RR^N$, $h(x) = \frac12 \| Hx - y \|^2$.

We consider a restoration example in image processing. Precisely, we choose $\overline{x}$ to be the image \texttt{jetplane} of size $N = 256 \times 256$ shown in Figure~\ref{Fig:simul_pb}, and $H \in [0,+\infty[^{N \times N}$ to model a blurring operator. In this case, $H$ is implemented as a convolution operator such that the image is convolved with a motion blurring kernel of length 5 and angle $60^\circ$. 
The noisy observation is then obtained building $b$ as a realization of an i.i.d. Gaussian variable with zero-mean and standard deviation $\sigma>0$. In our simulations we will consider two different noise levels, defined through the input signal-to-noise ratio (iSNR): $\text{iSNR} = \log_{10} \| H \overline{x} \|^2 / (N \sigma^2)$.
Precisely, we will consider the cases when iSNR~$=20$~dB and iSNR~$=25$~dB. In these cases, the SNR of the observed images $y$ are equal approximately to $18$~dB and $21$~dB, respectively. For each case, we run simulations for 50 realizations of random noise.
In Figure~\ref{Fig:simul_pb} are shown the blurred image $H\overline{x} \in \RR^N$ and an example of a noisy observation $y \in \RR^N$, when iSNR~$=25$~dB. 

For the matrices $(A_{k,i})_{k \in \NN, \, 0 \le i \le I_k-1}$, since the preconditioning scheme has been already widely discussed in the literature (see e.g. \cite{Chouzenoux13, Chouzenoux_2016, Combettes_Vu_2014, Repetti_SPL_2015}), we will not investigate this functionality in our simulations. 
In our simulations, we propose to use the same approach as proposed in \cite{Chouzenoux13}. In this context, the matrices are diagonal, fixed over iterations, and can be seen as diagonal approximation of the Hessian of $h$. 

We will consider two different regularization terms: the log-sum penalization and the $\ell_\rho^\rho$ penalization described in sections~\ref{Sssec:logsum} and \ref{Ssec:lpnorm}, respectively. 
For the two considered penalization terms we will provide reconstruction results obtained using a classic VMFB algorithm \cite{Chouzenoux13}\footnote{Note that the VMFB algorithm does not have convergence guaranties when considering $g$ to be a non-convex function \cite{Chouzenoux13}, unlike the FB algorithm \cite{Attouch_Bolte_2011}. Nevertheless, the convergence guaranties can be deduced from \cite{Chouzenoux_2016}, for standard numerical implementations of the log function.} and the proposed C2FB algorithm. 
Since C2FB is proved to converge for any value of $I_k \in \NN^*$, we will consider different values of $I_k \equiv I$, keeping it fixed over iterations. 
As explained with equation~\eqref{algo:FB_MM_approx}, when $I$ is chosen large enough (i.e. when the inner-loop has converged), C2FB has a similar behaviour as state-of-the-art methods to minimize composite functions (see e.g. \cite{Drusvyatskiy2016, Geiping2018, Ochs2015siam, Ochs2018JOTA}). 
%
%
%
For both the methods we will investigate the convergence behaviour, the reconstruction quality of the estimate, and the convergence speed. 
For the convergence behaviour, we will compare, for each experiment, the value of the objective function at convergence. Precisely, we will evaluate $C(x^\star_{VMFB}, x^\star_{C2FB}) :=  \Big( f(x^\star_{VMFB}) - f(x^\star_{C2FB}) \Big) / |f(x^\star_{VMFB})|$, where $x^\star_{VMFB}$ and $x^\star_{C2FB}$ are the estimates obtained with VMFB and C2FB, respectively. If this criteria is positive, it means that our method reached a better critical point than VMFB, i.e. $ f(x^\star_{VMFB}) > f(x^\star_{C2FB}) $.
To evaluate the reconstruction quality of the estimate, we use the signal-to-noise ratio (SNR), which is defined, for an image $x \in \RR^N$ as $\text{SNR} = 10 \log_{10} \Big( \|\overline{x}\|^2 / \|\overline{x}-x\|^2 \Big)$.
Finally, we will investigate the convergence speed of both the methods, in terms of number of iterations needed to reach convergence.
We consider that both the algorithms have converged when the following stopping criteria are fulfilled:
\begin{equation}    \label{eq:stop_crit}
    \begin{cases}
    \| x_k - x_{k+1} \| < 10^{-6} \| x_{k+1} \|,   \\
    \| f(x_k) - f(x_{k+1}) \| < 10^{-5} \| f(x_{k+1}) \|,
    \end{cases}
\end{equation}
where $(x_k)_{k\in \NN}$ is the sequence generated either by VMFB or by the proposed C2FB. It is important to emphasize that, when the computation of the proximity operator does not require sub-iterations, one iteration of VMFB has similar computational cost as one inner-iteration of C2FB. 
For the three above-mentioned evaluation criteria, we will show that the proposed approach leads to better results than the VMFB algorithm. 
In addition, choosing different values of $I$ (number of inner-iterations) will allow us to show that there is an optimal value for $I$ in terms of convergence speed and reconstruction quality, suggesting that there is no need for reaching convergence in the inner-iterations before recomputing the majorant function.

\begin{figure}
    \centering
    \includegraphics[width=3cm]{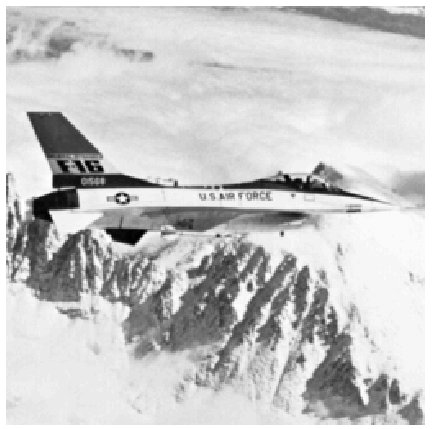}
    \includegraphics[width=3cm]{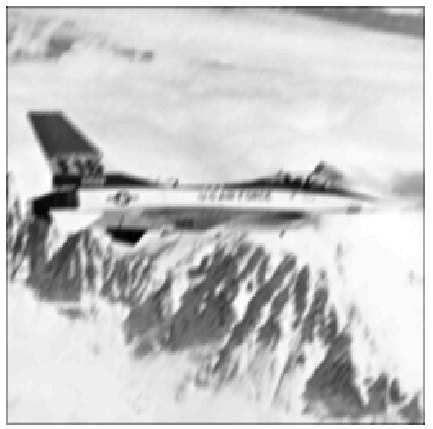}
    \includegraphics[width=3cm]{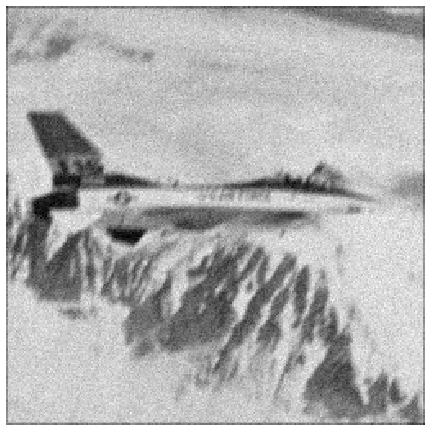}

    \vspace*{-0.3cm}
    
    \caption{\label{Fig:simul_pb}\footnotesize 
    Image \texttt{jetplane} used for the simulations. From left to right: Original unknown image $\overline{x}$, blurred image $H\overline{x}$, and noisy observation $y$ for an iSNR of $25$ dB. The SNR between the original image $\overline{x}$ and the observations $y$ is equal to $20.9$ dB.
    }
\end{figure}

\subsection{Log-sum penalization}
\label{Ssec:Simuls:RWl1}

In this section we present the simulation results obtained when solving the problem described above, using the log-sum regularization described in section~\ref{Sssec:logsum}. More precisely, we propose to 
\begin{equation}	\label{sim:pb:log-sum}
    \minimize{x\in \RR^N} \Big\{ f(x) := \frac12 \| Hx - y \|^2 + \theta \sum_{n=1}^N \log( |[Wx]^{(n)}| + \varepsilon) \Big\},
\end{equation}
where $W \colon \RR^N \to \RR^N$ models the Db8 wavelet transform \cite{Daubechies1998} with $4$ decomposition levels, and $\theta>0$ and $\varepsilon>0$ are chosen to maximize the reconstruction quality. In our experiments we have $(\theta,\varepsilon) = (10^8, 10^{-5})$ (resp. $(\theta,\varepsilon) = (3\times 10^8, 10^{-5})$) for iSNR~$=25$~dB (resp. $20$~dB). 
When \eqref{sim:pb:log-sum} is solved with the VMFB algorithm, we use the proximity operator of the log-sum function given in \eqref{eq:prox_logsum}.
When solved by our C2FB algorithm, we use algorithm~\eqref{algo:rL1_FB}, with  $I_k \equiv I \in \{2, \ldots, 250\}$.

Results are given in Figure~\ref{Fig:results_RL1}, considering (top row) iSNR~$=20$~dB and (bottom row) iSNR~$=25$~dB. 
The left plots give $C(x^\star_{VMFB}, x^\star_{C2FB})$ as a function of $I$. 
For the middle and right plots, the blue curves are obtained with C2FB, and the red curves are obtained considering VMFB, computing exactly the proximity operator.
The middle plots give the SNR values as a function of $I$; and the right plots show the total number of iterations needed to reach convergence as a function of $I$. For C2FB, the total number of iterations is given by $K^\star \times I$, where $K^\star$ is the number of outer iterations computed in algorithm~\eqref{algo:rL1_FB} to satisfy the stopping criteria~\eqref{eq:stop_crit}.
For VMFB, the curves are constant as there is no inner-loop in the algorithm. 
For both the methods, the continuous lines represent the average values, and the dotted lines show the associated results within 1 standard deviation around the mean. 

For both the considered noise levels, we observe that $C(x^\star_{VMFB}, x^\star_{C2FB})$ increases with $I$ and is always positive, showing that C2FB provides a better critical point. 
In addition, C2FB leads to better reconstruction results in terms of SNR than VMFB. When iSNR~$=20$~dB (resp. iSNR~$=25$~dB), the SNR obtained with C2FB is $\approx 22$~dB (resp. SNR~$\approx 23.6$~dB) when $I\ge 5$ (resp. $I \ge 10$), while VMFB leads to results with SNR~$=11$~dB (resp. SNR~$=15.8$~dB). 
In addition, C2FB necessitates less global iterations to reach convergence, and we observe that there is an optimal value for $I$ in terms of total iteration number. Precisely, when iSNR~$=20$~dB (resp. iSNR~$=25$~dB), the optimal value is around $I=15$ (resp. $I=60$), for a total number of iterations of $K^\star \times I = 165$ (resp. $K^\star \times I = 635$). In comparison, VMFB necessitates $\sim 2000$ iterations to converge in both the cases. This observation suggests that reducing the number of iterations in the inner-loop (instead of reaching convergence in each inner-loop before re-computing the weights, as suggested in classical reweighting $\ell_1$ algorithms \cite{candes2006compressive}) can accelerate the convergence of the reweighting $\ell_1$ algorithm without altering the reconstruction quality.

\begin{figure}
    \centering
    \begin{tabular}{c @{} c @{} c}
        \includegraphics[width=4.0cm]{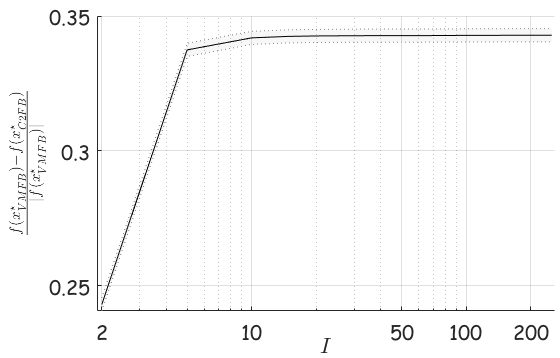}
    &   \includegraphics[width=4.0cm]{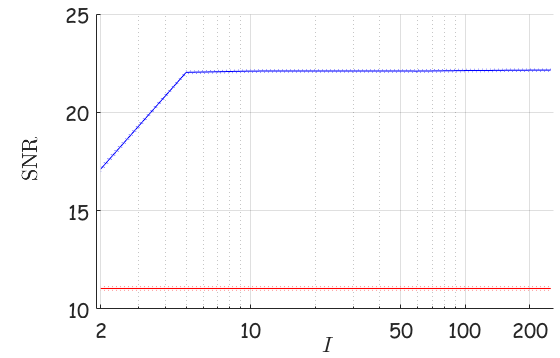}  
    &   \includegraphics[width=4.0cm]{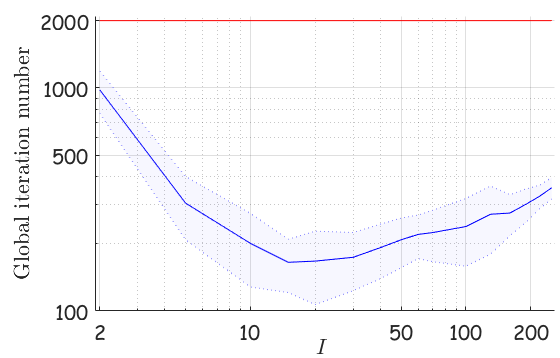}   \\[-0.1cm]
        \includegraphics[width=4.0cm]{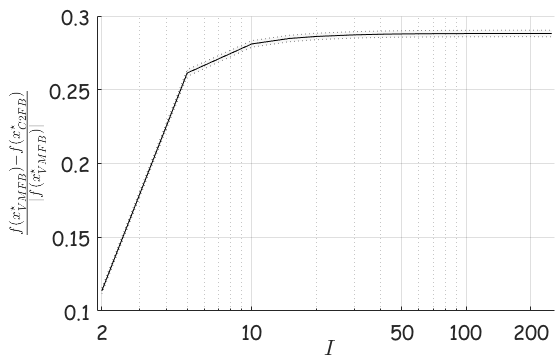}
    &   \includegraphics[width=4.0cm]{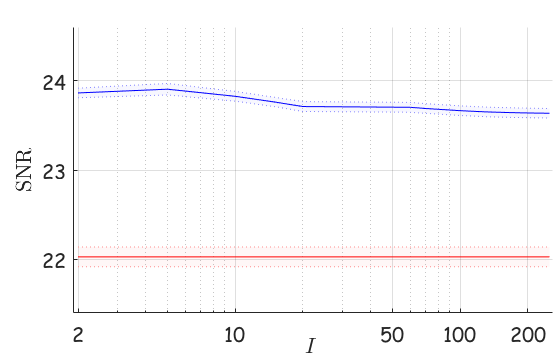}   
    &   \includegraphics[width=4.0cm]{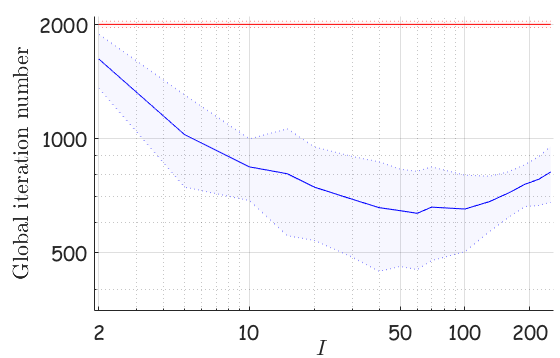}
    \end{tabular}
    
    \vspace{-0.4cm}
    
    \caption{\label{Fig:results_RL1}\footnotesize 
    Results for the log-sum penalization, considering a noise level with (top) iSNR $= 20$ dB (SNR of $y$ $ \approx 18$~dB), and (bottom) iSNR $= 25$ dB (SNR of $y$ $ \approx 21$~dB). 
    Comparison between the proposed reweighting $\ell_1$ algorithm~\eqref{algo:rL1_FB} considering different fixed numbers of iterations in the inner-loop $I_k \equiv I \in \{2, \ldots, 250\}$ (blue curves, using a log scale for the horizontal axis), and the VMFB algorithm with the exact proximity operator of the log-sum penalization computed as per \eqref{eq:prox_logsum} (red curves). 
    The red curves are constant since when the proximity operator is computed exactly, there is no inner-iterations for the reweighting.
    From left to right: 
    values of the objective function at convergence $f(x^\star)$ (linear scale);
	SNR values in dB (linear scale); 
	and global number of iterations needed to reach convergence (log scale). 
    For the proposed method, the global number of iterations needed to reach convergence corresponds to $K^\star \times I$, where $K^\star \in \NN$ is the number outer-iterations in algorithm~\eqref{algo:rL1_FB}. 
    The continuous lines are the average values obtained over 50 realizations of random noise, and the dotted lines show the associated results within 1 standard deviation around the mean. 
    }
\end{figure}

\subsection{$\ell_\rho^\rho$ penalization}
\label{Ssec:Simuls:lp}

\begin{figure}
    \centering
    \begin{tabular}{c @{} c @{} c}
        \includegraphics[width=4.0cm]{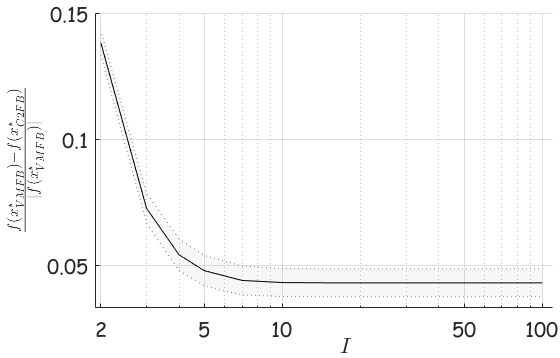}
    &   \includegraphics[width=4.0cm]{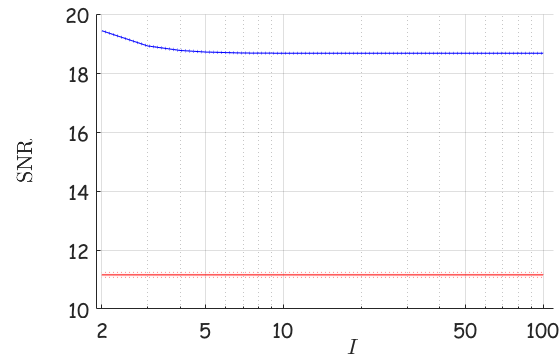}  
    &   \includegraphics[width=4.0cm]{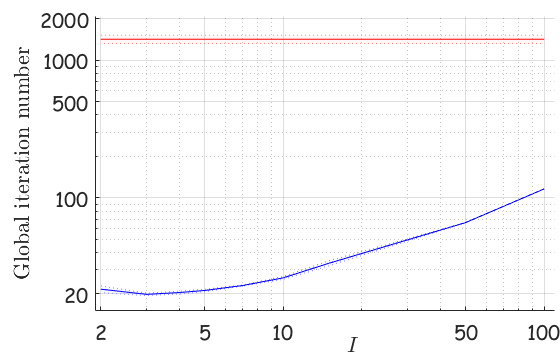}   \\[-0.1cm]
        \includegraphics[width=4.0cm]{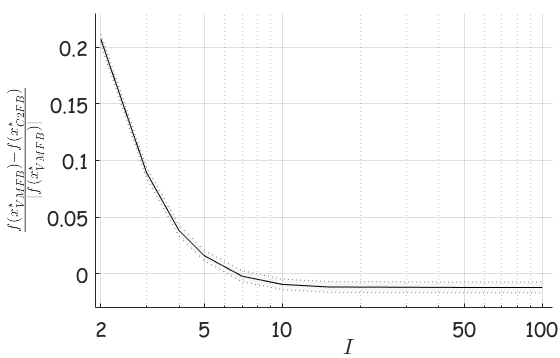}
    &   \includegraphics[width=4.0cm]{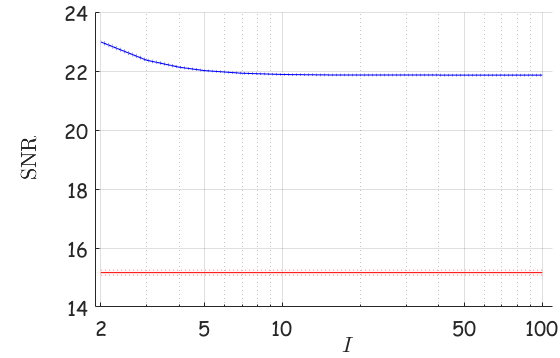}   
    &   \includegraphics[width=4.0cm]{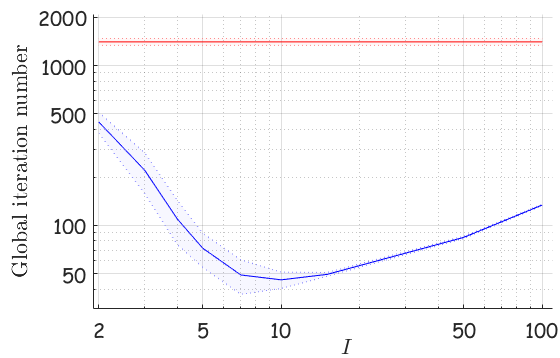}
    \end{tabular}
    
    \vspace{-0.4cm}
    
    \caption{\label{Fig:results_pwr}\footnotesize 
    Results for the $\ell_\rho^\rho$ penalization, with $\rho=10^{-3}$, considering a noise level with (top) iSNR $= 20$ dB (SNR of $y$ $ \approx 18$~dB) and (bottom) iSNR $= 25$ dB (SNR of $y$ $ \approx 21$~dB).
    Comparison between the proposed approximation algorithm~\eqref{algo:lp_FB} considering different fixed numbers of iterations in the inner-loop $I_k \equiv I \in \{2, \ldots, 100\}$ (blue curves, using a log scale for the horizontal axis), and the VMFB algorithm with the exact proximity operator of the $\ell_\rho^\rho$ penalization computed as per \eqref{eq:prox_lp} (red curves). Note that this proximity operator necessitates sub-iterations to be computed. The red curves are constant since when the proximity operator is computed exactly, there is no inner-iterations for the reweighting.
    From left to right: 
    values of the objective function at convergence $f(x^\star)$ (linear scale);
	SNR values in dB (linear scale); 
	and global number of iterations needed to reach convergence (log scale). 
    For the proposed method, the global number of iterations needed to reach convergence corresponds to $K^\star \times I$, where $K^\star \in \NN$ is the number of outer-iterations in algorithm~\eqref{algo:lp_FB}. 
    The continuous lines are the average values obtained over 50 realizations of random noise, and the dotted lines show the associated results within 1 standard deviation around the mean. 
    }
\end{figure}

In this section we present the simulation results obtained when solving the problem described at the beginning of the section, using the $\ell_\rho^\rho$ regularization described in section~\ref{Ssec:lpnorm}, with $\rho=10^{-3}$. More precisely, we propose to 
\begin{equation}    \label{eq:pb_l01}
    \minimize{x\in \RR^N} \Big\{ f(x) := \frac12 \| Hx - y \|^2 + \theta \widetilde{\ell}_{10^{-3},\varepsilon}(Wx)^{10^{-3}} \Big\},
\end{equation}
where $\widetilde{\ell}_{10^{-3}, \varepsilon}$ is the approximation of the $\ell_{10^{-3}}$ norm defined in \eqref{def:mod_ellp}, $W \colon \RR^N \to \RR^N$ models the Db8 wavelet transform \cite{Daubechies1998}, with $4$ decomposition levels, and $\theta>0$ and $\varepsilon>0$ are chosen to maximize the reconstruction quality. In our experiments we have $(\theta, \varepsilon) = (2\times 10^2, 10^{-5})$ (resp. $(\theta,\varepsilon) = (10^3, 10^{-5})$) for iSNR~$=25$~dB (resp. $20$~dB).
Problem~\eqref{eq:pb_l01} can be solved using our C2FB algorithm, with algorithm~\eqref{algo:lp_FB}. In our experiments, we run C2FB for different number of inner-iterations $I \in \{2, \ldots, 100\}$.
In the limit case when $\varepsilon=0$, then $\widetilde{\ell}_{\rho,\varepsilon}$ reduces to the exact $\ell_{\rho}$ norm, and problem~\eqref{eq:pb_l01} can be solved using the VMFB algorithm, where the proximity operator of the $\ell_{\rho}$ norm is given by \eqref{eq:prox_lp}. Note that even if the proximity operator of the $\ell_{\rho}$ norm has an explicit formula, in practice it requires to compute sub-iterations. 
%

Results are given in Figure~\ref{Fig:results_pwr}, with (top row) iSNR~$=20$~dB and (bottom row) iSNR~$=25$~dB. 
The left plots show $C(x^\star_{VMFB}, x^\star_{C2FB})$ as a function of $I$. 
For the middle and right plots, the blue curves are obtained with C2FB, and the red curves are obtained using VMFB where the proximity operator is computed using equation~\eqref{eq:prox_lp} combined with a Newton method. 
The central (resp. middle) plots show SNR values (resp. total iteration number needed to reach convergence) as a function of $I$. 

In this experiment, the conclusions are slightly different depending on the considered noise level.
For the noise level corresponding to iSNR~$=20$~dB (top row), the criteria $C(x^\star_{VMFB}, x^\star_{C2FB})$ is decreasing when $I$ increases, and is positive for all values of $I$, showing that the proposed C2FB algorithm provides a better critical point than VMFB independently from the number of inner-iterations. 
In particular, the best results are obtained when taking $I=2$. This conclusion is also true when observing the middle and right plots. For the SNR (see middle plots), C2FB leads to results with SNR~$=19.4$~dB for $I=2$, decreasing to SNR~$=18.7$~dB when $I\ge 5$. In comparison, the SNR for the estimate obtained with VMFB is equal to $11.2$~dB. On the right plots, we observe that the total iteration number needed to reach convergence increases with $I$ for C2FB, starting at $K^\star \times I = 20$ for $I=2$ and finishing at $K^\star \times I = 116$ for $I=100$. For comparison, VMFB necessitates $1423$ iterations to reach convergence. It is worth noticing that for this penalization function, unlike our method, the VMFB algorithm necessitates sub-iterations to compute the proximity operator.
For the noise level corresponding to iSNR~$=25$~dB (bottom row), the criteria $C(x^\star_{VMFB}, x^\star_{C2FB})$ decreases when $I$ increases, and is positive for $I \le 5$. This shows that the critical point obtained with C2FB is better than the one obtained with VMFB only for a small number of inner-iterations $I$. For the reconstruction quality (see middle plots), the results obtained with C2FB have a higher SNR, equal to $23$~dB for $I=2$, and decreasing to $22$~dB for $I\ge 5$. For comparison, the SNR of the estimate obtained with VMFB is equal to $15.2$~dB. Finally, for the computational cost (see right plots), we observe that the total number of iterations $K^\star \times I$ needed by C2FB to reach convergence, is decreasing for $I \le 10$ and increasing for $I\ge 10$. When $I=2$ (i.e. corresponding to the best reconstruction, both in terms of objective value and reconstruction quality), C2FB needs $K^\star \times I \approx 440$ iterations to reach convergence. This total iteration number drops to $45$ for $I=10$. For comparison, VMFB requires $1400$ iterations to reach convergence (without counting the sub-iterations to compute the proximity operator).
We can conclude that the proposed approach provides a good alternative to the classic VMFB method, in terms of both quality reconstruction and convergence speed. In addition, the proposed C2FB algorithm outperforms as well state-of-the-art methods to minimize composite functions, obtained when $I \to \infty$ (see \cite{Drusvyatskiy2016, Geiping2018, Ochs2015siam, Ochs2018JOTA}).

\bibliographystyle{plain}
\bibliography{abbr,references}

\end{document}